\documentclass[12pt]{amsart}
\newtheorem{theorem}{Theorem}[section]
\newtheorem{lemma}[theorem]{Lemma}
\newtheorem{proposition}[theorem]{Proposition}
\theoremstyle{definition}
\newtheorem{definition}[theorem]{Definition}

\theoremstyle{remark}
\newtheorem{remark}[theorem]{Remark}

\numberwithin{equation}{section}

\usepackage{tikz-cd}
\usepackage{graphicx,geometry}%
\usepackage{multirow}%
\usepackage{amsmath,amssymb,amsfonts}%
\usepackage{amsthm}%
\usepackage{mathrsfs}%
\usepackage[title]{appendix}%
\usepackage{xcolor}%
\usepackage{textcomp}%
\usepackage{manyfoot}%
\usepackage{booktabs}%
\usepackage{algorithm}%
\usepackage{algorithmicx}%
\usepackage{algpseudocode}%
\usepackage{listings}%

\allowdisplaybreaks[1]
\usepackage{epstopdf}
\usepackage{booktabs}
\usepackage{graphicx}
\usepackage{subfigure}
\renewcommand\figurename{Fig.}
\usepackage{mathtools}
\def\bu{\boldsymbol{u}}
\def\d{\operatorname{d}}

\def\v{\boldsymbol{v}}
\def \w {\boldsymbol{w}}
\def\vh{\boldsymbol{v}_h}

\def\div{\operatorname{div}}
\def\curl{\operatorname{curl}}
\def \nf{\boldsymbol{n}_f}
\def\wh{\boldsymbol{w}_h}

\def \Hgd{\boldsymbol{H}_0(\operatorname{grad-div};\Omega)}
\def \Vrk{\boldsymbol{V}_{r-1, k+1}(K)}

\begin{document}
	\begin{center}
		\begin{LARGE}
			{The grad-div conforming virtual element method for the quad-div problem in three dimensions}
		\end{LARGE}
	\end{center}
	
	\title{}

	\author{Xiaojing Dong$^1$}

	\address{$^1$Hunan Key Laboratory for Computation and Simulation in Science and Engineering, Key Laboratory of Intelligent Computing \& Information Processing of Ministry of Education, School of Mathematics and Computational Science, Xiangtan University, Er huan Road, Xiangtan,  411105, Hunan, P.R. China}

	\email{dongxiaojing99@xtu.edu.cn}

	\thanks{The research was supported by the National Natural Science Foundation of China (Nos: 12071404, 12571393, 12261131501), National Natural Science Foundation of China Key Program (No: 12431014), Young Elite Scientist Sponsorship Program by CAST (No: 2020QNRC001), the Science and Technology Innovation Program of Hunan Province (No: 2024RC3158), Postgraduate Scientific Research Innovation Project of Hunan Province (No: LXBZZ2024112), the Project of Scientific Research Fund of the Hunan Provincial Science and Technology Department (No.2023GK2029, No.2024ZL5017), and Program for Science and Technology Innovative Research Team in Higher Educational Institutions of Hunan Province of China.}

	\author{Yibing Han$^{1\dagger}$}
	\thanks{$^\dagger$ Corresponding author.}
	
	\email{202331510114@smail.xtu.edu.cn}
	
	\author{Yunqing Huang$^1$}
	\email{huangyq@xtu.edu.cn}
	
	
	\begin{abstract}
		We propose a new stable variational formulation for the quad-div problem in three dimensions and prove its well-posedness. 
		Using this weak form, we develop and analyze the $\boldsymbol{H}(\operatorname{grad-div})$-conforming virtual element method of arbitrary approximation orders on polyhedral meshes. 
		Three families of $\boldsymbol{H}(\operatorname{grad-div})$-conforming virtual elements are constructed based on the structure of a de Rham sub-complex with enhanced smoothness, resulting in an exact discrete virtual element complex.
		In the lowest-order case, the simplest element has only one degree of freedom at each vertex and face, respectively. 
		The rigorous analysis includes interpolation error estimates,  stability of the discrete bilinear forms, well-posedness of the discrete formulation, and optimal convergence rates.  Some numerical examples are shown to verify the theoretical results.
	\end{abstract}
	

	\maketitle

\section{Introduction}
\label{intro}
Let $\Omega \subset \mathbb{R}^3$ be a contractible Lipschitz polyhedron with boundary $\Gamma$ and unit outward normal $\boldsymbol{n}$.
We consider the following quad-div problem: find $\bu$ such that
\begin{equation}\label{primalForm}
	\begin{aligned}
		(\nabla \div)^2 \boldsymbol{u} &=\boldsymbol{f} \quad\text {in } \Omega, \\
		\operatorname{curl} \,  \boldsymbol{u} &=0 \quad\text{in } \Omega, \\
		\boldsymbol{u} \cdot \boldsymbol{n} &=0 \quad\text{on } \Gamma, \\
		\div \boldsymbol{u}  &=0 \quad\text {on } \Gamma,
	\end{aligned}
\end{equation}
where $\boldsymbol{f}$ is a curl-free vector field.
The quad-div operator arises in linear elasticity \cite{Altan1992,Mindlin1964,Mindlin1965}, with $\boldsymbol{u}$ representing the displacement field of an elastic body and the integral of $(\nabla \operatorname{div})^2$ over the domain corresponding to the shear strain energy.
This operator can be expressed as $(\operatorname{div}^* \circ \operatorname{div})^* \circ (\operatorname{div}^* \circ \operatorname{div})$, a form it shares with other fundamental fourth-order operators such as the biharmonic $\Delta^2$ and the quad-curl $\operatorname{curl}^4$, all belonging to the class $(D^*\circ D)^* \circ (D^* \circ D)$. It is the dual of the biharmonic operator and, in two dimensions, shares properties similar to those of the quad-curl operator. While the biharmonic and quad-curl operators have been extensively studied \cite{QuadCurlHodge,QuadCurlXuLiWei,Bihar2002,ZZMQuadCurl,HuangXHQuadCurl,WGQuadCurl,ChenHuangZhang2025}, research on the quad-div operator remains limited \cite{QuadDiv2018,QuadDiv2022,ZZMQuadDiv}.
\par
The discrete de Rham complex plays a critical role in the design and analysis of finite element methods (FEMs), offering a structured framework for constructing finite element spaces that maintain compatibility with differential operators; see \cite{Arnold2018,ZZMQuadCurl,ZZMQuadDiv,HuangXHQuadCurl,Chen2024,ChenHuangZhang2025}. 
We introduce the following de Rham complex with enhanced smoothness and homogeneous boundary conditions:
\begin{equation}\label{ContiComplex}
	0 \stackrel{\subset}{\longrightarrow} H_0^{1}(\Omega) \stackrel{\nabla}{\longrightarrow} \boldsymbol{H}_0(\operatorname{curl};\Omega) \stackrel{\operatorname{curl}}{\longrightarrow}\boldsymbol{V}_0(\Omega) \stackrel{\operatorname{div}}{\longrightarrow} H_0^1(\Omega)\cap L^2_0(\Omega) {\longrightarrow} 0,
\end{equation}
where
$$\boldsymbol{V}_0( \Omega) := \{\boldsymbol{v} \in \boldsymbol{L}^2(\Omega) : \div \boldsymbol{v} \in H_0^1(\Omega) \text{ and } \v \cdot \boldsymbol{n}=0 \text{ on } \Gamma\}$$ 
denotes the homogeneous $\boldsymbol{H}(\operatorname{grad-div})$-conforming space. On contractible Lipschitz domains, the exactness of \eqref{ContiComplex} follows from the standard results in \cite{Arnold2018}.
\par 
The work of \cite{ZZMQuadDiv} exploits the framework \eqref{ContiComplex} to develop a conforming finite element discretization for the quad-div operator on $\boldsymbol{V}_0(\Omega)$. In that work, the authors also analyze the following problem:
\begin{equation}\label{quaddiv1}
	(\nabla\div)^2\bu +\bu =\boldsymbol{f}, \quad \curl \bu =\boldsymbol{g} \quad \text{in } \Omega,
\end{equation}
subject to the boundary conditions in \eqref{primalForm}, for given fields $\boldsymbol{f} \in \boldsymbol{H}(\curl;\Omega)$ and $\boldsymbol{g} \in \boldsymbol{L}^2(\Omega)$ where $\boldsymbol{g} = \curl \boldsymbol{f}$.
The corresponding weak formulation reads: find $\boldsymbol{u} \in \boldsymbol{V}_0(\Omega)$ such that
\begin{equation}\label{addu}
	(\nabla\nabla\cdot \bu ,\nabla\nabla\cdot\v) +(\bu,\v)=( \boldsymbol{f},\boldsymbol{v}), \quad \v \in \boldsymbol{V}_0(\Omega).
\end{equation}
To the best of our knowledge, a weak formulation for the problem \eqref{primalForm} remains unestablished when the curl-free source field $\boldsymbol{f}$ belongs to the dual of $\boldsymbol{V}_0(\Omega)$. The primary difficulty arises from the need to treat the kernels of both the divergence and the curl operator within the variational framework. In contrast,
this difficulty does not appear in the problem \eqref{addu} due to the presence of stabilization $(\bu,\v)$.
For the more regular case where the curl‑free field  $\boldsymbol{f}$ lies in the dual of $\boldsymbol{H}_0(\div;\Omega)$, a reduced-order variational formulation for \eqref{primalForm} has been investigated in \cite{QuadDiv2018}. 
\par 
Building upon the generalized Helmholtz decomposition \cite{ChenLong2018}, we characterize the dual of $\boldsymbol{V}_0(\Omega)$, which clarifies the admissible space for the source term $\boldsymbol{f}$.
Furthermore, we use the complex \eqref{ContiComplex} to formulate a new, equivalent weak form for \eqref{primalForm}, which is stabilized by introducing two Lagrange multipliers. As detailed in Section 3,  its well-posedness is proven via the Friedrichs inequality applied to divergence-free and curl-free spaces. Using a commutative diagram involving the associated trace operator, we define an appropriate trace space to extend \eqref{primalForm} to non-homogeneous boundary conditions. A crucial observation is that the divergence of the solution to the quad-div problem coincides with the unique solution of a Poisson problem subject to an integral constraint.
\par
The virtual element method (VEM) \cite{VEM2013} is a numerical technique for the approximation of partial differential equations (PDEs) that extends the finite element method (FEM) to general polygonal and polyhedral meshes.  
It combines the conformity of the finite element method with enhanced flexibility and simpler design. This is achieved by defining basis functions as solutions to local boundary value problems, without requiring explicit polynomial representations while simultaneously facilitating straightforward higher-order extensions \cite{VEM2013,VEMfordivcurl,Antonietti2014,VEMforStokes,VEMforMaxwell,ZJQDivFree,ZJQQuadCurl}. Our analysis of the non-homogeneous continuous problem leads to a new $\boldsymbol{H}(\operatorname{grad-div})$-conforming virtual element space, defined through the corresponding local boundary value problem, which paves the way for a conforming discretization.
\par 
In this paper, we construct the following discrete complex with integers $r,k\ge1$:
\begin{equation}\label{IntroDiscomplex}
	\mathbb{R} \stackrel{\subset}{\longrightarrow} U_{1}(\Omega) \stackrel{\nabla}{\longrightarrow} \boldsymbol{\Sigma}_{0,r}( \Omega) \stackrel{\curl}{\longrightarrow} \boldsymbol{V}_{r-1,k+1}( \Omega) \stackrel{\operatorname{div}}{\longrightarrow} W_k(\Omega) {\longrightarrow} 0.
\end{equation}
We start by introducing the subspaces $U_1(\Omega)$ and $W_k(\Omega)$ of  $H_1(\Omega)$ in Sec. 4.1. Subsequently, we introduce the $\boldsymbol{H}(\operatorname*{curl})$-conforming virtual element spaces $\boldsymbol{\Sigma}_{0,r}(\Omega)$ in Sec. 4.2.
Finally, taking $r=k$, $r=k+1$, $r=k+2$, we construct three families of $\boldsymbol{H}(\operatorname*{grad-div}$)-conforming virtual element spaces $\boldsymbol{V}_{r-1,k+1}(\Omega)$ in Sec. 4.3.
These families extend the $\boldsymbol{H}(\operatorname{grad-div})$-conforming finite elements introduced in \cite{ZZMQuadDiv} to general polyhedral meshes. In the lowest-order case ($r=k=1$), the construction reduces to the simplest finite element structure on simplicial and cuboid meshes, where the degrees of freedom are associated only with vertices and faces. 
Furthermore, in Sec. 4.4, we establish the exactness of the sequence from $\boldsymbol{V}_{r-1,k+1}(\Omega)$ to $\boldsymbol{W}_k(\Omega)$ by identifying the divergence of the quad-div solution with the unique solution of the corresponding nonhomogeneous Poisson problem subject to an integral constraint. Additionally, we define appropriate interpolation operators to obtain a commutative diagram connecting the discrete complex \eqref{IntroDiscomplex} with the continuous one.
\par 
Taking $r=k$ as an example, we provide the proof of the interpolation estimates and the stability of the discrete bilinear form defined on $\boldsymbol{V}_{k-1,k+1}(\Omega)$ in Sec. 4.5.
The remaining cases, such as $r=k+1$ or $r=k+2$, can be obtained similarly. 
Based on the above preparations, we present the discrete bilinear forms and derive their stability in Sec. 5.1.
Then we give the discrete formulation for the quad-div problem in Sec. 5.2.
Using discrete Friedrichs inequalities (see Lemma \ref{lemDisFre}), we prove the existence and uniqueness of the solution for the discrete system.
We show the optimal error estimates in Sec. 6. 
Finally, we display some numerical examples to verify the theoretical results.
\section{Notation}
Let $\{\mathcal{T}_h\}_h$ be a family decomposition of $\Omega $ into non-overlapping polyhedral elements $K$ with mesh size parameter $h$. 
For each element $K \in \mathcal{T}_h$, denote $\partial K$ by its boundary with unit outward normal $\boldsymbol{n}_{\partial K}$. For a face $f$ of $K$, let 
$\boldsymbol{n}_f$ be the restriction of $\boldsymbol{n}_{\partial K}$ to $f$, and for an edge $e$, let $\boldsymbol{t}_e$ be the unit tangential vector along $e$.
Given any geometric entity $G$ (element,  face, or edge), we denote by $h_G, |G|$ and $\boldsymbol{b}_G$ its diameter, measure and barycenter, respectively.
For all $ h:=\operatorname{max}_{K\in \mathcal{T}_h}\{h_K\}$,  given a uniform positive constant $\mu$,  we assume that every $K\in \mathcal{T}_h$ satisfies the following standard regularity assumptions \cite{VEMforGener,VEMforStokes,VEMforMaxwell}: \par                                                  
$\boldsymbol{(A1)}\, K$ is star-shaped with respect to a ball of radius $\ge\mu h_K$, \par
$\boldsymbol{(A2)}\, $every face $f$ of $K$ is star-shaped with respect to a disk  of radius $\ge \mu h_K$, \par
$\boldsymbol{(A3)}\, $every edge $e$ of $K$ and every face of $K$ satisfy $h_e\ge \mu h_f \ge \mu^2 h_K.$ \\
Throughout this paper,  we use the symbols $ \lesssim $ and $ \gtrsim $ to denote the upper and lower bounds up to a generic positive constant that is
independent of the discretization  parameters $h$, respectively.
\par
For a subdomain $G \subset \Omega$ with unit outward normal $\boldsymbol{n}_{\partial G}$ and a nonnegative real number $s$, the standard Sobolev spaces $H^s(G)$ and $H^s_0(G)$ are endowed with the norm $\|\cdot\|_{s,G}$ and seminorm $|\cdot|_{s,G}$. We denote
$H^0(G)$ by $L^2(G)$ equipped with the norm $\|\cdot \|_G $ and the inner product $(\cdot,  \cdot )_G$. The duality pairing is denoted by $\langle \cdot, \cdot \rangle_G$, and the corresponding dual space $H^{-s}(G)$ is defined, equipped with the norm $\|\cdot\|_{-s, G}$.
If $G=\Omega$,  we omit the subscript $G$. For $k \in \mathbb{N}_0$, we denote by $P_k(G)$ the space of polynomials of degree at most $k$ on $G$, with the convention $P_{-1}(G) = {0}$.  For $m < k$, define $\widehat{P}_{k/m}(G)$ to be the subspace of $P_k(G)$ spanned by monomials of degree greater than $m$.
We shall use $\boldsymbol{H}^s(G),  \boldsymbol{H}_0^s(G)$, $\boldsymbol{L}^2(G)$ and $\boldsymbol{P}_k(G)$ to denote the vector-valued Sobolev spaces $[H^s(G)]^3,  [H^s_0(G)]^3$,  $ [L^2(G)]^3$ and $[P_k(G)]^3$, respectively.
In a slight abuse of notation, we denote the gradient, Laplacian, curl, and divergence operators by $\nabla$, $\Delta$, $\operatorname{curl}$, and $\operatorname{div}$, respectively, while the Sobolev spaces $\boldsymbol{H}(\operatorname{curl};G)$, $\boldsymbol{H}_0(\operatorname{curl};G)$, $\boldsymbol{H}(\operatorname{div};G)$, and $\boldsymbol{H}_0(\operatorname{div};G)$ retain their standard definitions. We define the following function spaces: 
\begin{equation*}
	\begin{aligned}
		\boldsymbol{V}(G)&:=\{\boldsymbol{v}\in \boldsymbol{L}^2(G):\,  \div\boldsymbol{v}\in H^1(G)\}, \\
		\boldsymbol{V}_0(G)&:=\{ \boldsymbol{v}\in \boldsymbol{V}(G):\,  \boldsymbol{v}\cdot \boldsymbol{n}_{\partial G}=0 \text{ and } \div \boldsymbol{v} =0 
		\text{ on } \partial G \},
	\end{aligned}
\end{equation*}
equipped with the graph norm
\begin{align*}
	\|\boldsymbol{v}\|^2_{\boldsymbol{V}(G)}=\|\boldsymbol{v}\|^2_G+\|\operatorname{div}\boldsymbol{v}\|_{1,G}^2
\end{align*}
Note that $\boldsymbol{V}_0(G)$ coincides with the closure of $\boldsymbol{C}_0^{\infty}(G)$ with respect to the norm $\|\cdot\|_{\boldsymbol{V}(G)}$, which can be established by a standard density argument analogous to the proof of \cite[Theorem 2.3]{QuadCurlXuLiWei}.
The spaces of divergence-free and curl-free vector fields are defined by
\begin{equation*}
	\begin{aligned}
		\boldsymbol{H}(\operatorname{div}^0;\Omega)&:=\{\boldsymbol{v}\in \boldsymbol{L}^2(\Omega):  (\boldsymbol{v}, \nabla q)= 0,   \forall q \in H^1_0(\Omega)\}, \\
		\boldsymbol{H}(\operatorname{curl}^0;\Omega)&:=\{\boldsymbol{v}\in \boldsymbol{L}^2(\Omega):  (\boldsymbol{v}, \operatorname{curl} \boldsymbol{\phi})=0,  \forall \boldsymbol{\phi}\in \boldsymbol{H}_0(\operatorname{curl};\Omega)\}.
	\end{aligned}
\end{equation*}
Recall the $L^2$-orthogonal Helmholtz-Hodge decomposition \cite{Arnold2018}
\begin{align}\label{decom1}
	\boldsymbol{H}_0(\operatorname{curl};\Omega)&= \nabla H_0^1(\Omega) \oplus^\bot  \boldsymbol{X}(\Omega),  \\
	\boldsymbol{V}_0( \Omega)&= \operatorname{curl}\boldsymbol{H}_0(\operatorname{curl};\Omega)\oplus^\bot  \boldsymbol{Y}(\Omega),\label{decom2}
\end{align} 
where 
\begin{equation*}
	\begin{aligned}
		\boldsymbol{X}(\Omega)& = \boldsymbol{H}_0(\operatorname{curl};\Omega)\cap \boldsymbol{H}(\operatorname{div}^0;\Omega), \\
		\boldsymbol{Y}(\Omega)&=\boldsymbol{V}_0(\Omega)\cap \boldsymbol{H}(\operatorname{curl}^0;\Omega).
	\end{aligned}
\end{equation*}
We introduce the following result, based on the 
imbedding theory  \cite[Proposition 3.7]{Amrouche1998} and the norm equivalence \cite[Corollary 3.16 and 3.19]{Amrouche1998}.
\begin{lemma}
	If $\Omega$ is a contractible Lipschitz polyhedron,
	there exists $s>\frac{1}{2}$ such that  
	\begin{equation*}
		\begin{aligned}
			\|\boldsymbol{v}\|_s \lesssim    \|\operatorname{curl} \boldsymbol{v}\|+\|\operatorname{div}\boldsymbol{v}\|,  \quad \forall \boldsymbol{v} \in \boldsymbol{H}(\operatorname{curl};\Omega)\cap \boldsymbol{H}_0(\operatorname{div};\Omega) \\
			\text{ or }\boldsymbol{H}_0(\operatorname{curl};\Omega)\cap \boldsymbol{H}(\operatorname{div};\Omega).
		\end{aligned}
	\end{equation*}
	In particular, for divergence-free or curl-free fields,  the Friedrichs inequalities hold:
	\begin{align}
		\|\boldsymbol{v}\|_s &\lesssim  \|\operatorname{curl} \boldsymbol{v}\|,  \quad \forall \boldsymbol{v}\in \boldsymbol{X}(\Omega),   \label{Frieq2} \\
		\|\boldsymbol{v}\|_s &\lesssim  \|\operatorname{div} \boldsymbol{v}\|, \quad 
		\forall \boldsymbol{v} \in \boldsymbol{Y}(\Omega).  \label{Frieq3}
	\end{align}
	The regularity exponent $s$ typically satisfies $s = \frac{1}{2} $ for general Lipschitz domains, and improves to $s = 1$ when $\Omega$ is convex.
\end{lemma}
\par 
The following lemma characterizes the dual space  $\boldsymbol{V}'(\Omega)$ of $\boldsymbol{V}_0(\Omega)$. Its proof, which relies on \cite{Chen2018}, is provided in Appendix \ref{sec:appendixA}.
\begin{lemma}\label{DecLemma}
	The stable decomposition holds
	\begin{equation}\label{DualDec}
		\boldsymbol{V}'(\Omega)= \boldsymbol{H}^{-2}(\curl;\Omega)=\nabla H^{-1}(\Omega)\oplus \curl\boldsymbol{X}(\Omega) = \nabla H^{-1}(\Omega)\oplus \curl \boldsymbol{H}_0(\curl;\Omega),
	\end{equation}
	where the space
	\begin{equation*}
		\boldsymbol{H}^{-2}(\curl;\Omega):=\left\{\v \in \boldsymbol{H}^{-2}(\Omega): \curl \v \in \boldsymbol{H}^{-1}(\Omega) \right\},
	\end{equation*}
	equipped with the norm
	\begin{equation*}
		\|\v\|^2_{\boldsymbol{H}^{-2}(\curl;\Omega)}:=\|\v\|^2_{-2}+ \|\curl \v\|_{-1}^2.
	\end{equation*}
\end{lemma}
To characterize the trace space of $\boldsymbol{V}(\Omega)$, we first recall the classical trace operators. The standard theory states that on a Lipschitz domain $\Omega$, these operators extend continuously and surjectively to the relevant Sobolev spaces in the weak setting:
\begin{align}
	\label{Trace0}
	\gamma_0 &: H^s(\Omega) \rightarrow H^{s-1/2}(\Gamma): v \to v|_{\Gamma} \text{ for } \frac{1}{2}<s<\frac{3}{2} \quad \text{\cite[Theorem 3.38]{McLean2000}},  \\
	\boldsymbol{\gamma}_{\tau} &: \boldsymbol{H}(\text{curl};\Omega) \rightarrow \boldsymbol{H}^{-1/2}(\text{curl}_{\Gamma}; \Gamma): \v \to \boldsymbol{n}\wedge(\v \wedge\boldsymbol{n})|_{\Gamma} \quad \text{\cite[Theorem 4.1]{Buffa2002}},  \\
	\label{HDivTrace}
	\boldsymbol{\gamma}_n &: \boldsymbol{H}(\text{div};\Omega) \rightarrow H^{-1/2}(\Gamma):\v \to (\v\cdot \boldsymbol{n})|_{\Gamma} \quad \text{\cite[Theorem 2.5, Corollary 2.8]{GiraultRaviart1986}}. 
\end{align}
The trace space $H^s(\Gamma)$ for $0 \le s \le 1$ is defined via localization and pullback by charts \cite[Chapter 3]{McLean2000}. This definition is extended by duality to the range $-1 \le s < 0$, identifying $H^{-s}(\Gamma)$ as the dual of $H^s(\Gamma)$. The same chart-based approach also enables the definition of surface differential operators, including the surface gradient $\nabla_{\Gamma}$ and the surface scalar curl $\curl_{\Gamma}$, as described in \cite{Buffa2002,Buffa2003}. The trace space $\boldsymbol{H}^{-1/2}(\text{curl}_{\Gamma};\Gamma)$ 
denotes the dual of the range of the tangential trace operator applied to $\boldsymbol{H}^1(\Omega)$.
These mappings are summarized in the following commutative diagram: 
\begin{equation}\label{EnhaceDeRhamTraceCom}
	\begin{tikzcd}
		& 
		H^1(\Omega) \arrow[r, "\nabla"] \arrow[dd, "\gamma_0"] & 
		\boldsymbol{H}(\curl; \Omega)  \arrow[r, "\curl"] \arrow[dd, "\boldsymbol{\gamma}_{\tau}"] &
		\boldsymbol{V}(\Omega) \arrow[r, "\div"] \arrow[dd, "\boldsymbol{\gamma}_n"] \arrow[rd, dashed, "\boldsymbol{\gamma}_{\div}"'] & H^1(\Omega) \arrow[r,] \arrow[d, "\boldsymbol{\gamma}_0"] & 0\\
		&  &  &  & H^{\frac{1}{2}}(\Gamma) \arrow[d, ] \\
		& 
		H^{\frac{1}{2}}(\Gamma) \arrow[r, "\nabla_{\Gamma}"] & 
		\boldsymbol{H}^{-\frac{1}{2}}(\operatorname{curl}_{\Gamma};\Gamma) \arrow[r, "\operatorname{curl}_{\Gamma}"] & 
		H^{-\frac{1}{2}}(\Gamma)  \arrow[r,]& 0
	\end{tikzcd},
\end{equation}
where the operator $\boldsymbol{\gamma}_{\div}$ is defined by $\boldsymbol{\gamma}_{\div}(\v):= \gamma_0(\div \v)$ for any $\v \in \boldsymbol{V}(\Omega)$. Both complexes in \eqref{EnhaceDeRhamTraceCom} are exact on contractible Lipschitz domains \cite{Arnold2018}, which means:
\begin{gather}\label{ExactonDomain}
	\text{img}(\nabla)=\ker(\curl), \quad \text{img}(\curl)=\ker(\div),\quad \text{img}(\div) =H^1(\Omega),\\
	\label{ExactonTrace}
	\text{img}(\nabla_{\Gamma})=\ker(\curl_{\Gamma}),\quad \text{img}(\curl_{\Gamma})=H^{-\frac{1}{2}}(\Gamma).
\end{gather}
The trace space of $\boldsymbol{V}(\Omega)$ is defined as:
\begin{equation*}
	V(\Gamma):=\{(h, g)\in  H^{-\frac{1}{2}}(\Gamma)\times H^{\frac{1}{2}}(\Gamma)\}.
\end{equation*}

\begin{theorem}\label{SurjForV}
	On a contractible Lipschitz domain $\Omega$, the map $\v \to \{ \boldsymbol{\gamma}_{n}(\v), \boldsymbol{\gamma}_{\div}(\v)\} : \boldsymbol{V}(\Omega) \to V(\Gamma)$ is continuous and surjective.
\end{theorem}
\begin{proof}
	The continuity follows directly from \eqref{HDivTrace} and \eqref{Trace0}: for any $\v\in \boldsymbol{V}(\Omega)$,
	\begin{equation*}
		\| \boldsymbol{\gamma}_{\tau} (\v)\|_{-\frac{1}{2},\Gamma}+\|\boldsymbol{\gamma}_{\div}(\v)\|_{\frac{1}{2},\Gamma}\lesssim \|\v\|_{\boldsymbol{H}(\div;\Omega)}+ \|\div \v\|_{1} \lesssim \|\v\|_{\boldsymbol{V}(\Omega)}.
	\end{equation*}
	\par
	We now prove the surjectivity.
	Let $(h,g)\in V(\Gamma)$. By the surjectivity of $\gamma_0$, there exists $w \in H^1(\Omega)$ with $\gamma_0(w) = g$. The exactness \eqref{ExactonDomain} then yields a function $\boldsymbol{u}^g \in \boldsymbol{V}(\Omega)$ such that
	\begin{equation*} 
		\div\boldsymbol{\bu}^g=w.
	\end{equation*}
	Meanwhile, \eqref{ExactonTrace} implies the existence of $\boldsymbol{q}_{\tau} \in \boldsymbol{H}^{-\frac{1}{2}}(\curl_{\Gamma}; \Gamma)$ satisfying
	\begin{equation}\label{auxSurj}
		\curl_{\Gamma} \boldsymbol{q}_{\tau} =  h-\boldsymbol{\gamma}_{n}(\boldsymbol{\bu}^g).
	\end{equation}
	Taking $\boldsymbol{q} \in \boldsymbol{H}(\curl;\Omega)$ such that $\boldsymbol{\gamma}_{\tau}(\boldsymbol{q}) = \boldsymbol{q}_{\tau}$ (which exists by the surjectivity of $\boldsymbol{\gamma}_{\tau}$), and using the commutativity of \eqref{EnhaceDeRhamTraceCom} and \eqref{auxSurj}, we derive
	$$\boldsymbol{\gamma}_{n}(\curl \boldsymbol{q})=  \curl_{\Gamma}\boldsymbol{\gamma}_{\tau}(\boldsymbol{q})= h-\boldsymbol{\gamma}_{n}(\boldsymbol{\bu}^g). $$
	Thus, we define  $\boldsymbol{u}^{\partial}: = \curl \boldsymbol{q} + \boldsymbol{u}^g$, which belongs to $\boldsymbol{V}(\Omega)$ and meets the required boundary conditions:
	\begin{equation*}
		\boldsymbol{\gamma}_{n}(\boldsymbol{\boldsymbol{u}}^{\partial})= \boldsymbol{\gamma}_{n}(\curl \boldsymbol{q})+\boldsymbol{\gamma}_{n}(\boldsymbol{u}^{g})=h \text{ and }\boldsymbol{\gamma}_{\div}(\boldsymbol{\bu}^{\partial})=\boldsymbol{\gamma}_0(\div \boldsymbol{u}^{\partial})=\boldsymbol{\gamma}_0(\div \boldsymbol{u}^{g})=g.
	\end{equation*}
	The proof is complete.
\end{proof}
\section{The quad-div problem}
We begin by deriving the variational formulation for \eqref{primalForm}. From the characterization \eqref{DualDec}, for a source term $\boldsymbol{f}\in \boldsymbol{V}'(\Omega)\cap \ker(\curl)$, we have
\begin{equation*}
	\boldsymbol{V}'(\Omega)\cap \ker(\curl) = \boldsymbol{H}^{-2}(\curl;\Omega)\cap\ker(\curl)=\boldsymbol{H}^{-2}(\Omega)\cap\ker(\curl).
\end{equation*}
Then the dual norm of $\boldsymbol{f}$ satisfies 
\begin{equation}\label{dualNorm}
	\|\boldsymbol{f}\|_{\boldsymbol{V}'(\Omega)}=\|\boldsymbol{f}\|_{\boldsymbol{H}^{-2}(\curl;\Omega)}=\|\boldsymbol{f}\|_{-2}.
\end{equation}
A natural first attempt is to introduce a Lagrange multiplier $\boldsymbol{\varphi}$, leading to the following mixed formulation of \eqref{primalForm}: 
find $(\bu,\boldsymbol{\varphi})\in \boldsymbol{V}_0(\Omega)\times \boldsymbol{H}_0(\curl;\Omega)$ such that
\begin{equation}\label{illmix}
	\begin{aligned}
		(\nabla\div\bu,\nabla\div\v)+(\curl\boldsymbol{\varphi},\boldsymbol{v})&=\langle\boldsymbol{f},\boldsymbol{v}\rangle,\quad \forall \v \in \boldsymbol{V}_0(\Omega),\\
		(\bu,\curl\boldsymbol{\phi})&=0,\quad \forall \boldsymbol{\phi}\in \boldsymbol{H}_0(\curl;\Omega).    
	\end{aligned}
\end{equation}
However, due to the non-trivial kernel of the curl operator, the substitution $\v=\curl\boldsymbol{\varphi}$ gives only $\curl \boldsymbol{\varphi}=0$, which does not imply  $\boldsymbol{\varphi}=0$. Consequently, the problem \eqref{illmix} is not well‑posed, and an additional 
restriction $\boldsymbol{\varphi} \in \boldsymbol{H}_0(\curl;\Omega)/\nabla H_0^1(\Omega)$ is required. 
Thus, we define the weak formulation of the quad-div problem \eqref{primalForm} as:
\begin{definition}
	Given $\boldsymbol{f}\in \boldsymbol{H}^{-2}(\Omega)\cap\ker(\curl)$,  find 
	$(\boldsymbol{u},  \boldsymbol{\varphi},  p ) \in  \boldsymbol{V}_0(\Omega)\times\boldsymbol{H}_0(\operatorname{curl};\Omega)\times H_0^1(\Omega)$ such that
	\begin{equation}\label{mixForm}
		\begin{aligned}
			(\nabla\div \boldsymbol{u}, \nabla\div \boldsymbol{v})+(\operatorname{curl}\boldsymbol{\varphi}, \boldsymbol{v})&=\langle\boldsymbol{f}, \boldsymbol{v}\rangle, \quad \forall \boldsymbol{v}\in \boldsymbol{V}_0(\Omega), \\
			(\boldsymbol{u}, \operatorname{curl}\boldsymbol{\phi})+(\nabla p,  \boldsymbol{\phi}) &=0,  \quad \forall \boldsymbol{\phi}\in \boldsymbol{H}_0(\operatorname{curl};\Omega), \\
			(\boldsymbol{\varphi}, \nabla q  )&=0, \quad \forall q\in H_0^1(\Omega).
		\end{aligned}
	\end{equation}
\end{definition}
\begin{theorem}\label{Theorem1}
	Assume that $\Omega$ is a contractible Lipschitz domain.
	For each $\boldsymbol{f}\in\boldsymbol{H}^{-2}(\Omega)\cap \ker(\curl)$,  there exists a unique solution $(\boldsymbol{u},  \boldsymbol{\varphi},  p )$ $\in \boldsymbol{V}_0(\Omega)\times\boldsymbol{H}_0(\operatorname{curl};\Omega)\times H_0^1(\Omega)$ 
	for the problem \eqref{mixForm}. Moreover,  $\boldsymbol{\varphi}=0 $,  $p = 0$ and $\boldsymbol{u}$ satisfies
	\begin{equation*}
		\|\boldsymbol{u}\|_{\boldsymbol{V}(\Omega)} \lesssim  \|\boldsymbol{f}\|_{-2}.
	\end{equation*}
\end{theorem}
\begin{proof}
	For any  $\boldsymbol{v} \in \boldsymbol{Y}(\Omega) $,  we use the Poincar\'{e} inequality and the Friedrichs inequality \eqref{Frieq3} to get 
	\begin{equation}\label{acoer}
		(\nabla\div \v,\nabla\div \v)\gtrsim \|\div \v\|^2_1\gtrsim \|\v\|^2_{\boldsymbol{V}(\Omega)}.
	\end{equation}
	Define a bounded bilinear form $B:( \Hgd \times H_0^1(\Omega) ) \times \boldsymbol{H}_0(\operatorname{curl};\Omega) \to \mathbb{R}$ by 
	$B((\boldsymbol{v}, q), \boldsymbol{\phi}) = (\boldsymbol{v}, \operatorname{curl}\boldsymbol{\phi})+(\nabla q, \boldsymbol{\phi})$. Let
	\begin{equation*}
		\boldsymbol{Z}(\Omega)=\{(\boldsymbol{v}, q)\in \boldsymbol{V}_0(\Omega)\times H^1_0(\Omega): B((\boldsymbol{v}, q), \boldsymbol{\phi})=0,  \forall \boldsymbol{\phi}\in \boldsymbol{H}_0(\operatorname{curl};\Omega)\}.
	\end{equation*}
	For any $(\boldsymbol{v}, q) \in \boldsymbol{Z}(\Omega)$, taking $\boldsymbol{\phi} = \nabla q$ gives $B((\boldsymbol{v}, q), \boldsymbol{\phi}) = (\nabla q, \nabla q) = 0$, which implies $q = 0$. Hence, each $(\boldsymbol{v}, q) \in \boldsymbol{Z}(\Omega)$ is identified with $\boldsymbol{v} \in \boldsymbol{Y}(\Omega)$ and $q = 0$. This identification allows us to introduce the bilinear form $$A((\boldsymbol{u}, p), (\boldsymbol{v}, q) ):=(\nabla \operatorname{div}\boldsymbol{u}, \nabla \operatorname{div}\boldsymbol{v}),$$ which, together with \eqref{acoer}, yields the coercivity estimate
	\begin{equation*}
		A((\boldsymbol{v}, q), (\boldsymbol{v}, q)) \gtrsim(\|\boldsymbol{v}\|^2_{\boldsymbol{V}(\Omega)}+\|q\|_1^2), \quad \forall (\boldsymbol{v}, q)\in \boldsymbol{Z}(\Omega).
	\end{equation*}
	For any $\boldsymbol{\phi} \in \boldsymbol{H}_0(\operatorname{curl};\Omega)$,  
	the orthogonal decomposition \eqref{decom1} leads to $\boldsymbol{\phi} = \nabla \lambda +\boldsymbol{z} $,  where $\lambda \in H_0^1(\Omega)$ and $\boldsymbol{z} \in \boldsymbol{X}(\Omega)$. 
	Choosing $q=\lambda $ and $ \boldsymbol{v} =\operatorname{curl} \boldsymbol{\phi} $ yields
	\begin{equation*}
		\begin{aligned}
			B((\boldsymbol{v}, q), \boldsymbol{\phi})&=(\operatorname{curl}\boldsymbol{z}, \operatorname{curl}\boldsymbol{\phi})+(\nabla \lambda, \boldsymbol{\phi})\\ &=(\operatorname{curl}\boldsymbol{z}, \operatorname{curl}\boldsymbol{z})+(\nabla \lambda, \nabla \lambda)\\
			&\gtrsim \|\operatorname{curl}\boldsymbol{z}\|^2 + \|\boldsymbol{z}\|^2+|\lambda|_1^2 \\
			& = \|\boldsymbol{\phi}\|^2_{\boldsymbol{H}(\operatorname{curl};\Omega)}.
		\end{aligned}
	\end{equation*}
	Thus, the coercivity on $\boldsymbol{Z}(\Omega)$ and 
	Babu$\check{\text{s}}$ka-Brezzi condition are satisfied for the following variational problem 
	\begin{equation*}
		\begin{aligned}
			A((\bu,p),(\v,q))+B((\v,q),\boldsymbol{\varphi}) &= \langle\boldsymbol{f}, \boldsymbol{v}\rangle,\quad \forall (\v,q) \in \boldsymbol{V}_0(\Omega)\times H^1_0(\Omega),\\
			B((\bu,p),\boldsymbol{\phi}) &=0, \quad \forall \boldsymbol{\phi} \in \boldsymbol{H}_0(\operatorname{curl}; \Omega),
		\end{aligned}
	\end{equation*}
	which is obviously equivalent to the variational problem \eqref{mixForm}.
	Consequently,  \eqref{mixForm} has a unique solution and
	\begin{equation*}
		\|\boldsymbol{u}\|_{\boldsymbol{V}(\Omega)}+\|p\|_1+\|\boldsymbol{\varphi}\|_{\boldsymbol{H}(\operatorname{curl};\Omega)} \lesssim  \|\boldsymbol{f}\|_{\boldsymbol{V}'(\Omega)}=\|\boldsymbol{f}\|_{-2}.
	\end{equation*}
	We have used the dual norm \eqref{dualNorm}.
	\par
	Replacing $\boldsymbol{\phi} $ with $\nabla p $ in the second equation of \eqref{mixForm}, we obtain $p=0$ from the Poincar\'{e} inequality. According to the decomposition \eqref{decom1}, there exist $\lambda \in H^1_0(\Omega)$ and $\boldsymbol{z}\in \boldsymbol{X}(\Omega)$ such that $ \boldsymbol{\varphi} = \nabla \lambda + \boldsymbol{z}$.  
	Taking $\boldsymbol{v} = \curl \boldsymbol{\varphi} $ and $q =\lambda$ in the first equation and the last equation of \eqref{mixForm}, 
	we get $(\curl \boldsymbol{\varphi}, \curl \boldsymbol{\varphi}) = (\curl \boldsymbol{z}, \curl \boldsymbol{z}) = 0$ and $(\boldsymbol{\varphi}, \nabla\lambda)= (\nabla\lambda ,\nabla\lambda) =0$, respectively.
	Using the Friedrichs inequality \eqref{Frieq2} and Poincar\'{e} inequality,  we have $\boldsymbol{z}=0$ and $\lambda=0$. 
	Thus, it follows that
	\begin{equation}\label{equivalentProblem}
		\begin{aligned}
			(\nabla \operatorname{div}\boldsymbol{u}, \nabla \operatorname{div}\boldsymbol{v})&=\langle\boldsymbol{f}, \boldsymbol{v}\rangle, \quad \forall \boldsymbol{v}\in \boldsymbol{V}_0(\Omega),\\
			(\boldsymbol{u}, \operatorname{curl}\boldsymbol{\phi})&=0, \quad \quad \quad 
			\forall \boldsymbol{\phi}\in \boldsymbol{H}_0(\operatorname{curl};\Omega), 
		\end{aligned}
	\end{equation}
	which completes the proof.
\end{proof}
The density of $\boldsymbol{C}_0^{\infty}(\Omega)$ in $\boldsymbol{V}(\Omega)$ and $\boldsymbol{H}_0(\curl\Omega)$ implies that the solution $\bu\in \boldsymbol{Y}(\Omega)$ to the mixed formulation \eqref{mixForm} satisfies the primal formulation \eqref{primalForm} in the sense of distributions.
\begin{remark}[The regularity of the solution]\label{EquAndReg}
	On the contractible polyhedral domain $\Omega$, the Friedrichs inequality \eqref{Frieq3} implies that
	$$\bu \in \boldsymbol{H}^s(\Omega) \text{ with } s>\frac{1}{2}.  $$
	From the exactness of the sequence \eqref{Realscomplex}, for any given $\boldsymbol{f} \in \boldsymbol{H}^{s-2}(\Omega)\cap\ker(\curl)$  with $s\ge 0$, there exists a function $j_0 \in H^{s-1}(\Omega)$ such that $\boldsymbol{f} = -\nabla j_0$. Applying integration by parts to both sides of the first equation in \eqref{equivalentProblem} then yields 
	\begin{equation}\label{divPos}
		-\Delta \div \bu = j_0/\mathbb{R}.    
	\end{equation}
	By elliptic regularity on polyhedral domains \cite[Theorem 3.18]{Monk2003}, we have $$\operatorname{div} \boldsymbol{u} \in H^{s+1}(\Omega),$$ provided that $j_0 \in H^{s-1}(\Omega)$ with $s > \frac{1}{2}$.  
\end{remark}
\begin{remark}\label{divSolution}
	Consider the following Poisson equation with an integral constraint: given $j\in H^{-1}(\Omega)$, find $(w, C) \in H_0^1(\Omega) \times \mathbb{R}$ such that
	\begin{equation}\label{HomPos}
		\begin{aligned}
			(\nabla w, \nabla v)+ (C,v)&=\langle j,v \rangle,\quad \forall v \in H_0^1(\Omega),\\
			(w,q)&=0,\quad \forall q\in \mathbb{R}.
		\end{aligned}
	\end{equation}
	The well-posedness of this problem follows directly from the classical theory of mixed variational formulations \cite{Boffi2013}. 
	Furthermore, assuming $\boldsymbol{f} = -\nabla j$, let $\boldsymbol{u}$ and $w$ be the unique solutions of \eqref{mixForm} and \eqref{HomPos}, respectively. The exactness $\operatorname{div} \boldsymbol{V}_0(\Omega) = H_0^1(\Omega) \cap L_0^2(\Omega)$ implies
	\begin{equation*}
		(\nabla\div \bu ,\nabla \div \v)= \langle j, \div \v \rangle = (\nabla w, \nabla \div \v), \quad \forall \v \in \boldsymbol{V}_0(\Omega),
	\end{equation*}
	from which the Poincar\'{e} inequality yields $\operatorname{div} \boldsymbol{u} = w$.
\end{remark}
\par
The construction of the $\boldsymbol{H}(\operatorname{grad-div})$-conforming virtual element space relies on the following boundary value problem:
\begin{definition} 
	Given 
	\begin{equation}\label{Nonhomcond}
		\boldsymbol{f}^d\in \boldsymbol{H}^{-2}(\Omega)\cap\ker(\curl), \boldsymbol{f}^c\in \boldsymbol{H}^{-1}(\Omega)\cap\ker(\div), h\in H^{-\frac{1}{2}}(\Gamma) \text{ and } g\in H^{\frac{1}{2}}(\Gamma),    
	\end{equation}
	find $\bu \in \boldsymbol{V}(\Omega)$ with $(\boldsymbol{\gamma}_{n}(\v),\boldsymbol{\gamma}_{\div}(\v))=(h,g)\in Y(\Gamma)$ such that 
	\begin{equation}\label{NonHomQuadDiv}
		(\nabla\div)^2\bu^{\text{non}}=\boldsymbol{f}^d,\quad \curl \bu^{\text{non}} =\boldsymbol{f}^c  \text{ in } \Omega.
	\end{equation}
\end{definition}
We consider the non‑homogeneous Poisson problem:  given $j\in H^{-1}(\Omega)$, find $(w^{\text{non}},C)\in H^1(\Omega)\times \mathbb{R}$ such that  
\begin{equation}\label{EquH1problem}
	-\Delta w^{\text{non}} +C = j \text{ in } \Omega,\quad  w^{\text{non}} = g \text{ on } \Gamma, \quad (w^{\text{non}},1)= \langle h,1\rangle_{\Gamma},
\end{equation}
where $g$ and $h$ are given data, identical to those specified in \eqref{Nonhomcond}. 
In analogy with Remark \ref{divSolution}, we interpret the solution of \eqref{EquH1problem} as the divergence component of the solution to \eqref{NonHomQuadDiv}.
\begin{theorem}\label{NonHomTheorem}
	Let $\Omega$ be a contractible Lipschitz domain. Then the quad-div problem \eqref{NonHomQuadDiv} is well-posed, with its unique solution $\boldsymbol{u}^{\text{non}} \in \boldsymbol{V}(\Omega)$ satisfying
	\begin{equation}\label{NonhomEsti}
		\|\bu^{\text{non}}\|_{\boldsymbol{V}(\Omega)}\lesssim \|\boldsymbol{f}^d\|_{-2}+ \|\boldsymbol{f}^c\|_{-1} +\|h\|_{-\frac{1}{2},\Gamma} +\|g\|_{\frac{1}{2},\Gamma}.
	\end{equation}
	The Poisson problem \eqref{EquH1problem} is likewise well-posed, and if $\boldsymbol{f}^d = -\nabla j$, we have $$\operatorname{div} \boldsymbol{u}^{\mathrm{non}} = w^{\mathrm{non}}$$ for its unique solution $(w^{\mathrm{non}}, C) \in H^1(\Omega) \times \mathbb{R}$.
\end{theorem}	
\begin{proof}
	By Theorem \ref{SurjForV}, there exists a vector field $\boldsymbol{u}^{\partial} \in \boldsymbol{V}(\Omega)$ satisfying the boundary conditions
	$$\boldsymbol{\gamma}_{n}(\bu^{\partial})=h \quad \text{and}\quad  \boldsymbol{\gamma}_{\div}(\bu^{\partial})=g,$$ 
	and the estimate
	\begin{equation}
		\|\bu^{\partial}\|_{\boldsymbol{V}(\Omega)}\lesssim \|h\|_{-\frac{1}{2},\Gamma} + \|g\|_{\frac{1}{2},\Gamma}.
	\end{equation}
	Then the problem \eqref{NonHomQuadDiv} can be decomposed into finding $\boldsymbol{u}^{0} \in \boldsymbol{Y}(\Omega)$ and $\boldsymbol{\psi} \in \boldsymbol{X}(\Omega)$ that satisfy, respectively,
	\begin{equation}\label{problem1}
		(\nabla\div \bu^{0},\nabla\div \boldsymbol{v})= \langle\boldsymbol{f}^d,\boldsymbol{v}\rangle-(\nabla\div \bu^{\partial},\nabla\div\boldsymbol{v}),\quad \forall\boldsymbol{v}\in\boldsymbol{Y}(\Omega),
	\end{equation}
	and
	\begin{equation}\label{problem2}
		(\curl \boldsymbol{\psi},\curl \boldsymbol{\phi})=\langle\boldsymbol{f}^c ,q\rangle-(\bu^{\partial},\curl \boldsymbol{\phi}),\quad \forall \boldsymbol{\phi}\in \boldsymbol{X}(\Omega).
	\end{equation}
	Applying Friedrichs inequalities \eqref{Frieq3} and \eqref{Frieq2} together with the Poincaré inequality yields the estimates
	\begin{equation*}
		\|\bu^{0}\|_{\boldsymbol{V}(\Omega)}\lesssim |\div \bu^{0}|_{1} \lesssim \|\boldsymbol{f}^d\|_{-2} + |\div\bu^{\partial}|_{1},
	\end{equation*}
	and 
	\begin{equation*}
		\|\curl \boldsymbol{\psi}\|_{\boldsymbol{V}(\Omega)}=\|\curl\boldsymbol{\psi}\|\lesssim\|\boldsymbol{f}^c\|_{-1} + \|\bu^{\partial}\|.
	\end{equation*}
	Consequently, the function defined by $\boldsymbol{u}^{\text{non}} := \boldsymbol{u}^{0} + \curl \boldsymbol{\psi} + \boldsymbol{u}^{\partial}$ is the unique solution of \eqref{NonHomQuadDiv} and satisfies \eqref{NonhomEsti}. Furthermore, if $w^0$ is the unique solution of \eqref{HomPos}, then we identify $w^{\text{non}} = w^0 + \div \boldsymbol{u}^{\partial}$ as the unique solution to \eqref{EquH1problem}.
	\par
	Now, assume $\boldsymbol{f}^d = -\nabla j$. Then, for all $\boldsymbol{v} \in \boldsymbol{V}_0(\Omega)$, we have
	\begin{align*}
		(\nabla\div \bu^{\text{non}},\nabla\div \boldsymbol{v}) = \langle (\nabla\div)^2\bu^{\text{non}},\boldsymbol{v}\rangle=\langle \boldsymbol{f}^d,\boldsymbol{v} \rangle\\
		=\langle j,\div \v\rangle =(\nabla w^{\text{non}},\nabla\div\v).
	\end{align*}
	Moreover, the following boundary and integral conditions hold:
	\begin{equation*}
		\div \bu^{\text{non}} = w^{\text{non}} \text{ on } \Gamma,\quad ( \div \bu^{\text{non}}, 1)= \langle\bu\cdot \boldsymbol{n},1 \rangle_{\Gamma}=\langle h,1 \rangle_{\Gamma} = (w^{\text{non}},1).
	\end{equation*}
	Thus, we deduce that $\div \boldsymbol{u} ^{\text{non}} = w ^{\text{non}}$ holds, thereby completing the proof.
\end{proof}
\section{Virtual element spaces}
Based on the continuous complex
\begin{equation}\label{GraddivComplex}
	\mathbb{R} \stackrel{\subset}{\longrightarrow} H^{1}(\Omega) \stackrel{\nabla}{\longrightarrow} \boldsymbol{H}(\curl;\Omega) \stackrel{\curl}{\longrightarrow}\boldsymbol{V}(\Omega) \stackrel{\div }{\longrightarrow} H^1(\Omega) \longrightarrow 0,
\end{equation}
we construct a conforming discrete subcomplex for any integers $r$, $k\ge 1$:
\begin{equation}\label{discomplex}
	\mathbb{R} \stackrel{\subset}{\longrightarrow} U_{1}(\Omega) \stackrel{\nabla}{\longrightarrow} \boldsymbol{\Sigma}_{0,r}( \Omega) \stackrel{\text { curl }}{\longrightarrow} \boldsymbol{V}_{r-1,k+1}( \Omega) \stackrel{\operatorname{div}}{\longrightarrow} W_k(\Omega) {\longrightarrow} 0.
\end{equation}
To this end, the $\boldsymbol{H}(\operatorname{grad-div})$-conforming virtual element space $\boldsymbol{V}_{r-1,k+1}(\Omega) \subset \boldsymbol{V}(\Omega)$ is carefully designed to satisfy $$\div \boldsymbol{V}_{r-1,k+1}(\Omega) = W_k(\Omega).$$  
In addition, we establish a commutative diagram that links the continuous and discrete complexes. We also derive the corresponding interpolation error estimates and stability analysis. \par
Central to the VEM framework is the definition of appropriate projection operators onto polynomial spaces.
On any given element or face $G$,  we introduce the $H^1$- seminorm projection $\Pi_k^{\nabla, G} : H^1(G) \to P_k(G)$ with $k\in \mathbb{N}_0$:
\begin{equation*}
	\left\{
	\begin{aligned}
		&(\nabla \Pi_k^{\nabla, G}v-\nabla v, \nabla p_k)_G=0,  \forall p_k \in P_k(G),  \\
		&\int_{\partial G} \Pi_k^{\nabla, G}v \d s=\int_{\partial G} v \d s. 
	\end{aligned}
	\right.
\end{equation*}
Moreover,  we define the $L^2$- projection $\Pi_k^{0, G}$: $L^2(G) \to P_k(G)$ on $G$ by 
\begin{equation*}
	(\Pi^{0, G}_k v-v, p_k)_G=0,  \forall p_k\in P_k(G),
\end{equation*}
which  can be easily extended to vector function $\Pi_{k}^{0,G}: \boldsymbol{L}^2(G)\to \boldsymbol{P}_k(G)$.
For polynomial vector spaces,  we have the following useful decompositions \cite{VEMfordivcurl,VEMforMaxwell}
\begin{align}
	\boldsymbol{P}_k(G)&=\nabla P_{k+1}(G)\oplus \boldsymbol{x}\wedge \boldsymbol{P}_{k-1}(G),  \label{Pkdc1}\\
	\boldsymbol{P}_k(G)&=\operatorname{curl} \boldsymbol{P}_{k+1}(G) \oplus \boldsymbol{x}P_{k-1}(G).  \label{Pkdc2}
\end{align}
Furthermore, the following identity holds:
\begin{equation}\label{VectPolyIden}
	\boldsymbol{P}_{r-2}(K)\cap\ker(\div)=\curl \boldsymbol{P}_{r-1}(K)=\curl (\boldsymbol{x}\wedge \boldsymbol{P}_{r-1}(K))=\boldsymbol{P}_{r-1}(K)/(\nabla P_{r}(K)),
\end{equation}
with the curl operator being an isomorphism on $\boldsymbol{x} \wedge \boldsymbol{P}_{k-1}(K)$. 
\subsection{$H^1$-conforming virtual element spaces}
In this subsection, we present two different  $H^1$-conforming virtual spaces.
For each face $f \in \partial K$, we recall the face space \cite{VEM2013} with $r\ge 1$, $l\ge 1$ 
\begin{equation}\label{Bf}
	\mathbb{B}_{l,r}(f):=\left\{v \in H^1(f): \Delta v \in P_{r-2}(f),  v|_e \in P_{l}(e), \forall e\in \partial f, v|_{\partial f} \in C^0(\partial f)\right\},
\end{equation}
equipped with the degrees of freedom
\begin{flalign}
	&\	\bullet \mathbf{D}^1_{\mathbb{B}}:  \text{the values of } v \text{ at the vertices of } f, \label{DB1}&\\
	&\	\bullet \mathbf{D}^2_{\mathbb{B}}:  \text{the values of } v \text{ at }l-1 \text{ distinct points of } e, \label{DB2}&\\
	&\	\bullet \mathbf{D}^3_{\mathbb{B}}:  \text{the face moments } \frac{1}{|f|}\int_f v  p_{r-2} \d f,  \forall p_{r-2} \in P_{r-2}(f) , \label{DB3}
\end{flalign}
Note that $\mathbf{D}^3_{\mathbb{B}}$ can be alternatively expressed as \cite{VEMforMaxwell}
\begin{equation}\label{AlDof}
	\frac{1}{|f|}\int_f \nabla v \cdot \boldsymbol{x}_f p_{r-2} \d f,  \forall p_{r-2} \in P_{r-2}(f),
\end{equation}
where $\boldsymbol{x}_f:= \boldsymbol{x}-\boldsymbol{b}_f$ satisfies $\int_f \boldsymbol{x}_f \d f=0$.
To reduce the number of degrees of freedom, we introduce the serendipity space from \cite{VEMforlowMaxwell}
\begin{equation}\label{SerBf}
	\begin{aligned}
		\overline{\mathbb{B}}_{1}(f):=\left\{v \in \mathbb{B}_{1,2}(f):
		\int_f\nabla v \cdot \boldsymbol{x}_f \d f=0 \right\}.
	\end{aligned}
\end{equation}
This serendipity space is a subspace of the primal space $\mathbb{B}_{1,2}(f)$ in which the degrees of freedom \eqref{AlDof} vanish. This construction does not impact the polynomial completeness, as it still satisfies $P_1(f) \subset \mathbb{B}_{1,r}(f)$.
\par
The first local space on the polyhedral element $K$, used to approximate $p$, is taken to be the one defined in \cite{VEMforlowMaxwell}:
\begin{equation*}
	U_1(K):=\{ q \in H^1(K): \Delta q =0,  q|_f \in \overline{\mathbb{B}}_1(f), \forall f \in \partial K,  q|_{\partial K} \in C^0(\partial K)\}.
\end{equation*}
It is endowed only with the degrees of freedom 
\begin{flalign}
	&\	\bullet \mathbf{D}^1_{U}:  \text{the values of } q \text{ at the vertices of } K, &
\end{flalign}
and satisfies $P_1(K) \subset U_1(K)$.
Gluing the local space over all elements $K$ in $\mathcal{T}_h$ produces the global space 
\begin{equation*}
	U_{1}(\Omega):=\{ q \in H^1( \Omega): q|_K \in U_{1}(K), \forall K \in \mathcal{T}_h\}.
\end{equation*}
As will be shown in Theorem \ref{ExiUni} and confirmed by numerical experiments in the final section, the discrete solution $p_h \in U_1(\Omega)$ vanishes. Hence, higher-order spaces are unnecessary, and it suffices to work with the lowest-order space $U_1(\Omega)$.
\par
The other space corresponds to the construction of the $\boldsymbol{H}(\operatorname{grad-div}
)$-conforming space. 
The restricted $H^1$-conforming virtual element space of order $k\ge1$ on each face $f\in \partial K$ is given by
\begin{equation}
	\begin{aligned}
		\widehat{\mathbb{B}}_{k}(f):=\left\{q \in \mathbb{B}_{k,k+2}(f):
		(q-\Pi^{\nabla, f}_k q, \widehat{p}_k )_f=0,  \forall \widehat{p}_k  \in \widehat{P}_{k/k-2}(f)\right\},
	\end{aligned}
\end{equation}
where the projection $\Pi^{\nabla,f}_k$ is computable using the degrees of freedom \eqref{DB1}--\eqref{DB3} of the space $\mathbb{B}_{k,k}(f)$ (cf. \cite{Ahmad2013}).
This induces the boundary space
\begin{equation*}
	\mathbb{B}_{k}(\partial K):=\{q\in  C^0(\partial K): q|_f\in \widehat{\mathbb{B}}_{k}(f),  \forall f \in \partial K \}.
\end{equation*}
On the polyhedron $K$, we introduce the local enlarged space 
\begin{equation}
	\begin{aligned}
		\widetilde{W}_{k}(K):=\left\{q\in H^1(K): 
		\Delta q \in P_k(K),     q|_{\partial K} \in \mathbb{B}_{k}(\partial K)\right\},
	\end{aligned}
\end{equation}
and the final restricted space with $m=\text{max}\{0, k-2\}$
\begin{equation}\label{WFinal}
	W_k(K):=\left\{q\in \widetilde{W}_k(K): (q-\Pi^{\nabla, K}_k q, \widehat{p}_k )_K=0,   \forall  \widehat{p}_k \in P_{k/m}(K) \right\},
\end{equation}
equipped with the following degrees of freedom 
\begin{flalign}
	&\	\bullet \mathbf{D}^1_{W}:  \text{the values of } q \text{ at the vertices of } K, \label{DW1}&\\
	&\	\bullet \mathbf{D}^2_{W}:  \text{the values of } q \text{ at }k-1 \text{ distinct points of } e, \label{DW2}&\\
	&\	\bullet \mathbf{D}^3_{W}:  \text{the face moments } \frac{1}{|f|}\int_f qp_{k-2} \d f,  \forall p_{k-2} \in P_{k-2}(f), \label{DW3}&\\ 
	&\	\bullet \mathbf{D}^4_{W}:  \text{the volume moments } \frac{1}{|K|}\int_K q p_m \d K,  \forall p_m \in P_m(K).&	\label{DW4}
\end{flalign}
\begin{remark}
	As established in \cite{Ahmad2013}, 
	the enlarged space $\widetilde{W}_k(K)$ is equipped with the following additional degrees of freedom:
	\begin{equation*}
		\bullet \widetilde{\mathbf{D}}^4_{W}:  \text{the volume moments } \frac{1}{|K|}\int_K q \widehat{p}_k \d K,  \forall \widehat{p}_k \in P_{k/m}(K).
	\end{equation*}
	And  the projection $\Pi^{\nabla, K}_k$ is computable from \eqref{DW1}-\eqref{DW4}. 
	The space $W_k(K)$ is then identified as the subspace of $\widetilde{W}_k(K)$ where the values of  $\widetilde{\mathbf{D}}^4_{W}$ are constrained by the projection $\Pi_k^{\nabla,K}$. Furthermore, the dimension of $W_k(K)$ is given by
	\begin{equation}\label{dimWk}
		\begin{aligned}
			\dim(W_k(K))=\dim(\widetilde{W}_k(K))-\# \widetilde{\mathbf{D}}_{W}^4= \dim(\mathbb{B}_k(\partial K))+ \dim(P_m(K))\\
			=N_v+(k-1)N_e+\dim(P_{k-2}(f))N_f+\dim(P_m(K)).
		\end{aligned}
	\end{equation}
\end{remark}
\begin{remark}\label{EquWk}
	We consider an equivalent characterization of the space $\widetilde{W}_k(K)$:
	\begin{equation}\label{WideWk}
		\begin{aligned}
			\widehat{{W}}_k(K) := \bigg\{q \in H^1(K) :&\left\{
			\begin{aligned}
				&\Delta q + C \in \widehat{P}_{k/0}(K) \text{ for some constant } C, \\
				& \int_K q \mathrm{d} K \in P_0(K),
			\end{aligned}
			\right. \
			q|_{\partial K} \in \mathbb{B}_k(\partial K)
			\bigg\}.
		\end{aligned}
	\end{equation}
	The corresponding local problem is given by \eqref{EquH1problem}, with data $j\in \widehat{P}_{k/0}(K)$, $g\in \mathbb{B}_k(\partial K)$ and an integral constraint. It is straightforward to verify that the same degrees of freedom $\mathbf{D}_W^1$--$\mathbf{D}_W^4$ and $\widetilde{\mathbf{D}}_W^4$ as for $\widetilde{W}_k(K)$ are also unisolvent on $\widehat{W}_k(K)$. Since the two spaces have the same dimension and the inclusion  $\widetilde{W}_k(K)\subset\widehat{W}_k(K)$ holds, we conclude that $\widehat{W}_k(K)=\widetilde{W}_k(K)$.
	As will be established in Theorem \ref{DisExactTheorem}, the function defined in \eqref{WideWk} serves as the divergence component in the local $\boldsymbol{H}(\operatorname{grad-div})$-conforming virtual element space we aim to construct.
\end{remark}
As usual, we use the $H^1$-seminorm projection operator $\Pi^{\nabla, K}_k$ to discretize $(\nabla q_h, \nabla r_h )_K$ with $q_h, r_h \in W_k(K)$ as follows:
\begin{equation}
	[q_h,r_h]_{n,K}:=(\nabla\Pi_k^{\nabla,K}q_h, \nabla\Pi_k^{\nabla,K}r_h)_K+S_n^K\left((I-\Pi_k^{\nabla,K})q_h, (I-\Pi_k^{\nabla,K})r_h\right),
\end{equation}
where $I$ is the identity operator. The stabilizing bilinear $S_n^k(\cdot, \cdot )$ is chosen to be computable and needs to satisfy  
\begin{equation}\label{sta`b1}
	|q_h|^2_{1, K}\lesssim S_n^K(q_h, q_h) \lesssim |q_h|^2_{1, K}, \quad \forall q_h \in W_k(K) \cap \ker(\Pi^{\nabla,K}_k).
\end{equation}
\begin{remark}\label{remark2}
	In this paper, for all $q_h,  r_h \in W_k(K)$, we consider the following local stabilization 
	\begin{equation}\label{stab1}
		S_n^K(q_h, r_h)=h_K^{-2}(\Pi^{0, K}_kq_h, \Pi^{0, K}_kr_h)_K+\sum_{f\in \partial K}\left[h_f^{-1}(\Pi^{0, f}_kq_h, \Pi^{0, f}_kr_h)_f+(q_h, r_h)_{ \partial f}\right],
	\end{equation}
	whose coercivity and continuity, as stated in \eqref{sta`b1}, can be verified by an argument analogous to that in \cite[Theorem 2]{VEMhp}. Moreover, the lower bound established in \eqref{stab1} remains valid for all $q_h\in W_k(K)$, regardless of whether $q_h$ belongs to the kernel of $\Pi_{k}^{\nabla,K}$.
	According to the computable projections $\Pi^{\nabla, f}_k$ and $\Pi^{\nabla, K}_k$ and the degrees of freedom \eqref{DW3} and \eqref{DW4}, the projections $\Pi^{0, f}_k$ and $\Pi^{0, K}_k$ are computable.
	Combining the fact that $q_h$ and  $r_h$ belong to $P_k(e)$ on each edge, along with the degrees of freedom \eqref{DW1} and \eqref{DW2}, we easily compute the last term of \eqref{stab1}.
	As a result, we obtain the computability of the stabilization \eqref{stab1}.
\end{remark} 
We also define the global space by
\begin{eqnarray*}
	W_k(\Omega):=\{ w\in H^1(\Omega): w|_{K}\in W_k(K),  \forall K \in \mathcal{T}_h\}.
\end{eqnarray*}
\subsection{$\boldsymbol{H}(\operatorname{curl})$-conforming virtual element space}
In this subsection, we introduce a $\boldsymbol{H}(\operatorname*{curl})$-conforming virtual element space.
The edge virtual element  space \cite{VEMforlowMaxwell,VEMforMaxwell} on each face $f$ is as follows:                                                          
\begin{equation*}
	\begin{aligned}
		\mathbb{E}_{0, r}(f):=\left\{\boldsymbol{\phi} \in\left[L^{2}(f)\right]^{2}: \operatorname{div} \boldsymbol{\phi} \in P_{0}(f),  \operatorname{rot} \boldsymbol{\phi} \in P_{r-1}(f), \right. \\
		\left.\boldsymbol{\phi} \cdot \boldsymbol{t}_{e} \in P_{0}(e), \forall e \in \partial f,  \int_{f} \boldsymbol{\phi} \cdot \boldsymbol{x}_f\d f=0 \right\}, \\
	\end{aligned}
\end{equation*}
where the $\operatorname*{rot}$ operator is defined as $ \operatorname{rot} \boldsymbol{\phi}= \frac{\partial \phi_2}{\partial x_1} - \frac{\partial \phi_1}{\partial x_2}$.
On the element $K$, the local edge space  with $r\ge 2$  is given by \cite{VEMforMaxwell}
\begin{equation*}
	\begin{aligned}
		\boldsymbol{\Sigma}_{0, r}(K):=\left\{\boldsymbol{\phi} \in \boldsymbol{L}^2(K): \operatorname{div} \boldsymbol{\phi} =0,  \operatorname{curl}\operatorname{curl} \boldsymbol{\phi} \in\boldsymbol{P}_{r-2}(K),  \right. \boldsymbol{\phi}_{\tau}|_f \in \mathbb{E}_{0,r}(f),\\
		\left.  \forall f \in \partial K, \boldsymbol{\phi} \cdot \boldsymbol{t}_{e} \text { continuous on each edge } e \in \partial K\right\}, 
	\end{aligned}
\end{equation*}
with the following degrees of freedom 
\begin{flalign} \label{edgeDof1}
	&\	\bullet \mathbf{D}^1_{\boldsymbol{\Sigma}}: \text{the edge moments } \frac{1}{|e|}\int_e \boldsymbol{\phi}\cdot \boldsymbol{t}_e \d e, &\\
	&\	\bullet \mathbf{D}^2_{\boldsymbol{\Sigma}}: \text{the face moments } \frac{1}{|f|} \int_{f} \operatorname{rot} \boldsymbol{\phi}_{\tau} \widehat{p}_{r-1} \d f ,  \forall \widehat{p}_{r-1}\in \widehat{P}_{r-1/0}(f),&\\
	\label{edgeDof3}
	&\	\bullet \mathbf{D}^3_{\boldsymbol{\Sigma}}: \text{the volume moments  } \frac{1}{|K|}\int_K \operatorname{curl} \boldsymbol{\phi}\cdot( \boldsymbol{x}_K\wedge\boldsymbol{p}_{r-2} ) \d K, &\\
	\nonumber &\quad \quad \quad \quad \quad \quad \quad \quad \quad \quad \quad 
	\quad \quad \quad \quad \quad \quad \quad \, \, \, \forall \boldsymbol{p}_{r-2} \in \boldsymbol{P}_{r-2}(K), 
\end{flalign}
where $\boldsymbol{x}_K:=\boldsymbol{x}-\boldsymbol{b}_K$ and $\boldsymbol{\phi}_{\tau}$ denotes the tangential component of $\boldsymbol{\phi}$ on face $f$, that is, $\boldsymbol{\phi}_{\tau}=(\boldsymbol{\phi}-(\boldsymbol{\phi}\cdot \boldsymbol{n}_f)\boldsymbol{n}_f)|_f$.
As for $r=1$, we consider the  local serendipity space \cite{VEMforlowMaxwell}
\begin{equation}\label{SerSigma}
	\begin{aligned}
		\boldsymbol{\Sigma}_{0, 1}(K):=\left\{\boldsymbol{\phi} \in \boldsymbol{L}^2(K): \operatorname{div} \boldsymbol{\phi} =0,  \operatorname{curl}\operatorname{curl} \boldsymbol{\phi} \in \boldsymbol{P}_0(K), \boldsymbol{\phi}_{\tau} \in \mathbb{E}_{0,1}(f), \forall f \in \partial K, \right. \\
		\left. \boldsymbol{\phi} \cdot \boldsymbol{t}_{e} \text { continuous on each edge } e \in \partial K, \int_K\operatorname{curl}\boldsymbol{\phi}\cdot( \boldsymbol{x}_K\wedge \boldsymbol{p}_0) \d K=0, \forall \boldsymbol{p}_0 \in \boldsymbol{P}_{0}(K)\right\}. 
	\end{aligned}
\end{equation}
This space is constructed by first imposing the condition $\curl\curl \boldsymbol{\phi} \in \boldsymbol{P}_0(K)$, followed by the introduction of additional degrees of freedom:
\begin{equation}
	\int_K\operatorname{curl}\boldsymbol{\phi}\cdot( \boldsymbol{x}_K\wedge \boldsymbol{p}_0) \d K, \forall \boldsymbol{p}_0 \in \boldsymbol{P}_{0}(K).
\end{equation}
These additional degrees of freedom are subsequently eliminated. Thus, the resulting space \eqref{SerSigma} is endowed only with degrees of freedom \eqref{edgeDof1} and contains the lowest-order N\'ed\'elec  elements of the first kind.
The dimension of $\boldsymbol{\Sigma}_{0,r}(K)$ is derived as in \cite{VEMforMaxwell}:
\begin{equation*}
	\begin{aligned}
		\dim(\boldsymbol{\Sigma}_{0,r}(K))= N_e+ (\dim(P_{r-1}(f))-1)N_f+\dim(\boldsymbol{P}_{r-2}(K)\cap\ker(\div)),
	\end{aligned}
\end{equation*}
which, combined with the vector polynomial identity \eqref{VectPolyIden}
$$\dim(\boldsymbol{P}_{r-2}(K)\cap\ker(\div))=\dim(\boldsymbol{P}_{r-1}(K))-\dim(\nabla P_{r}(K)),$$
yields
\begin{equation}\label{dimSigma}
	\dim(\boldsymbol{\Sigma}_{0,r}(K)) =N_e+(\dim(P_{r-1}(f))-1)N_f+3\dim(P_{r-1}(K))-\dim(P_{r}(K))+1.
\end{equation}
\par
Note that the  degrees of freedom \eqref{edgeDof1}-\eqref{edgeDof3} on each element $K$ allow us to compute the 
$L^2$-orthogonal projection operator from $\boldsymbol{\Sigma}_{0, r}(K)$ to $\boldsymbol{P}_0(K)$, see \cite{VEMforlowMaxwell}. Following Remark \ref{remark2}, we 
define a discrete $L^2$- inner product by 
\begin{equation*}
	[\boldsymbol{\phi}_h,\boldsymbol{\psi}_h]_{e,K}:=(\Pi_{0}^{0,K}\boldsymbol{\phi}_h, \Pi_{0}^{0,K}\boldsymbol{\psi}_h)_K+S_e^K((I-\Pi^{0,K}_{0})\boldsymbol{\phi}_h, (I-\Pi_{0}^{0,K})\boldsymbol{\psi}_h),
\end{equation*}
which leads to 
\begin{equation}\label{stab2}
	\|\boldsymbol{\phi}_h\|^2_{ K}\lesssim S_e^K(\boldsymbol{\phi}_h, \boldsymbol{\phi}_h) \lesssim\|\boldsymbol{\phi}_h\|^2_{ K}, \quad \forall \boldsymbol{\phi}_h\in \boldsymbol{\Sigma}_{0,r}(\Omega) \cap \ker(\Pi^{0,K}_0).
\end{equation}
\begin{remark}
	For all $\boldsymbol{\phi}_h$, $\boldsymbol{\psi}_h \in \boldsymbol{\Sigma}_{0,r}(\Omega)$, we  choose the following stabilization from \cite{VEMforGener} with $r\ge 2$:
	\begin{equation*}
		\begin{aligned}
			S_e^K(\boldsymbol{\phi}_h,\boldsymbol{\psi}_h) &=h_K^2(\Pi^0_{\wedge, r-2}\curl\boldsymbol{\phi}_h, \Pi^0_{\wedge, r-2}\curl\boldsymbol{\psi}_h)_K\\
			&\quad+\sum_{f\in\partial K}[h_f^3(\operatorname{curl}\boldsymbol{\phi}_h \cdot \nf, \operatorname{curl} \boldsymbol{\psi}_h \cdot \nf )_f+ \sum_{e\in \partial f} h_f^2(\boldsymbol{\phi}_h\cdot\boldsymbol{t}_{e}, \boldsymbol{\psi}_h\cdot \boldsymbol{t}_{e})_{e}],
		\end{aligned}
	\end{equation*}
	and in the special case  $r=1$, following \cite{VEMforlowMaxwell}, we set 
	\begin{equation*}
		S_e^K(\boldsymbol{\phi}_h,\boldsymbol{\psi}_h) =h_K^2\sum_{e\in\partial K}(\boldsymbol{\phi}_h\cdot\boldsymbol{t}_{e}, \boldsymbol{\psi}_h\cdot \boldsymbol{t}_{e})_e.
	\end{equation*}
	Here, the computable projection $\Pi^0_{\wedge, r-2}: \boldsymbol{L}^2(K) \to \boldsymbol{x}_K\wedge \boldsymbol{P}_{r-2}(K)$ is defined as 
	\begin{equation*}
		(\Pi^0_{\wedge, r-2}\boldsymbol{\phi}-\boldsymbol{\phi}, \boldsymbol{p}_{r-1})_K=0,\quad \forall \boldsymbol{\phi}\in\boldsymbol{\Sigma}_{0,r}(K),\quad \forall \boldsymbol{p}_{r-1}\in \boldsymbol{x}_K\wedge\boldsymbol{P}_{r-2}(K).
	\end{equation*}
\end{remark}
Finally, we define the global space
\begin{equation*}
	\boldsymbol{\Sigma}_{0, r}(\Omega):=\{ \boldsymbol{\phi} \in \boldsymbol{H}(\text{curl}; \Omega): \boldsymbol{\phi}|_K \in \boldsymbol{\Sigma}_{0, r}(K), \forall K \in \mathcal{T}_h\}.
\end{equation*}
\subsection{$\boldsymbol{H}(\operatorname{grad-div})$-conforming virtual element space}
The $\boldsymbol{H}(\operatorname{grad-div})$-conforming virtual element space is constructed in this subsection. 
We begin with the local enlarged space for $r\ge 2$ and $k\ge 1$:
\begin{equation}\label{EnlargedV}
	\begin{aligned}
		\widetilde{\boldsymbol{V}}_{r-1,k+1}(K):=\bigg\{\boldsymbol{v} \in \boldsymbol{V}(K): 
		&\left\{
		\begin{aligned}
			&(\div \boldsymbol{v})|_{\partial K}\in \mathbb{B}_k(\partial K),\\
			&(\boldsymbol{v} \cdot \boldsymbol{n}_{\partial K})|_{\partial K}\in P_{r-1}(\partial K),
		\end{aligned}
		\right.\\
		&\left\{ \begin{aligned}
			&(\nabla\div)^2  \boldsymbol{v} \in \nabla P_{k}(K) , \\
			& \curl  \boldsymbol{v}\in \boldsymbol{P}_{r-2}(K)
		\end{aligned} \right. 
		\bigg\},
	\end{aligned}
\end{equation}
where $P_{r-1}(\partial K):=\{ v\in L^2(\partial K): v|_f \in P_{r-1}(f), \forall f \in \partial K\}$.
Its well-posedness follows directly from the well-posedness of the quad-div problem \eqref{NonHomQuadDiv} on the polyhedron $K$, with the following given data:
\begin{align*}
	\boldsymbol{f}^d\in\nabla P_k(K), \boldsymbol{f}^c \in \boldsymbol{P}_{r-2}(K)\cap(\ker(\div)), 
	(h,g)\in P_{r-1}(\partial K)\times \mathbb{B}_k(\partial K)\subset V(\partial K).
\end{align*}
Combining the dimension results from \eqref{dimWk} and \eqref{dimSigma}, we obtain the dimension of $\widetilde{\boldsymbol{V}}_{r-1,k+1}(K)$ as follows:
\begin{equation}\label{dimEnlargedV}
	\begin{aligned}
		\dim(\widetilde{\boldsymbol{V}}_{r-1,k+1}(K))&= \dim(\mathbb{B}_k(\partial K) + \dim(P_{r-1}(\partial K))+ \dim(\nabla P_k(K))+ \dim(\boldsymbol{P}_{r-2}(K)\cap(\ker(\div))) \\
		&=N_v+(k-1)N_e+(\dim(P_{k-2}(f))+\dim(P_{r-1}(f)))N_f+ \dim(P_k(K))-1 \\
		&\quad +3\dim(P_{r-1}(K))-\dim(P_{r}(K))+1.
	\end{aligned}
\end{equation}
Furthermore, for $r = 1$, a similar serendipity space is defined to satisfy the inclusion $\curl \boldsymbol{\Sigma}_{0,1}(K) \subset \widetilde{\boldsymbol{V}}_{0,k+1}(K)$:
\begin{equation}\label{Vr0}
	\begin{aligned}
		\widetilde{{\boldsymbol{V}}}_{0,k+1}(K):=\left\{ \v \in \widetilde{\boldsymbol{V}}_{r-1,k+1}(K): (\v\cdot \boldsymbol{n}_{\partial K})|_{\partial K} \in P_0(\partial K),\curl \v \in P_0(K),\right.\\ \left.\quad \int_K\v \cdot( \boldsymbol{x}_K\wedge \boldsymbol{p}_0) \d K=0, \forall \boldsymbol{p}_0 \in \boldsymbol{P}_{0}(K)\right\}.
	\end{aligned}
\end{equation}
As will be shown in Proposition \ref{UniProp}, its dimension coincides with that given by \eqref{dimEnlargedV} for the case $r=1$. 
\par 
From Remark \ref{EquWk}, we note that the quad-div problem defined in the enlarged space $\widetilde{\boldsymbol{V}}_{r-1,k+1}(K)$ corresponds to the Poisson problem in $\widetilde{W}_k(K)$, satisfying $\operatorname{div} \widetilde{\boldsymbol{V}}_{r-1,k+1}(K) \subset \widetilde{W}_k(K)$.  To ensure that the inclusion  $\operatorname{div} \boldsymbol{V}_{r-1,k+1}(K) \subset W_k(K)$ holds for the final restricted space with $r \ge 1$ and $k \ge 1$, we define
\begin{equation}\label{FinalSpace}
	\boldsymbol{V}_{r-1,k+1}(K):=\left\{ \v \in \widetilde{\boldsymbol{V}}_{r-1,k+1}(K): (\div \v-\Pi^{\nabla, K}_k \div \v, \widehat{p}_k )_K=0,   \forall  \widehat{p}_k \in \widehat{P}_{k/m}(K)\right\},
\end{equation}
where $m=\max\{0,k-2\}$. The computability of $\Pi_{k}^{\nabla,K} \operatorname{div}$ follows from Proposition \ref{L2Prop}. 
This inclusion is achieved by imposing the restriction in \eqref{FinalSpace}, which constitutes a divergence analogue of \eqref{WFinal}. 
\begin{proposition}\label{UniProp}
	For $r$, $k\ge 1$, the dimension of $\boldsymbol{V}_{r-1,k+1}(K)$ is given by
	\begin{equation}\label{DimVrk}
		\begin{aligned}
			\dim (\boldsymbol{V}_{r-1,k+1}(K) )
			&=N_v+(k-1)N_e+(\dim(P_{k-2}(f))+\dim(P_{r-1}(f)))N_f+ \dim(P_{m}(K)) \\
			&\quad +3\dim(P_{r-1}(K))-\dim(P_{r}(K)),
		\end{aligned}
	\end{equation}
	where $m=\max\{0,k-2\}$.
	Furthermore, the following degrees of freedom
	\begin{flalign}\label{dofV1}
		&\	\bullet \mathbf{D}^1_{\boldsymbol{V}}: \text{the values of }  \operatorname{div} \boldsymbol{v} \text{ at the vertices of } K, &\\
		\label{dofV2}
		&\ \bullet \mathbf{D}^2_{\boldsymbol{V}}:\text{the values of }\operatorname{div} \boldsymbol{v} \text{ at k-1 distinct points of every edge of } K,&\\
		\label{dofV3}
		&\ \bullet \mathbf{D}^3_{\boldsymbol{V}}:\text{the face moments } \frac{1}{|f|}\int_f \operatorname{div} \boldsymbol{v} p_{k-2} \d f,  \forall p_{k-2} \in P_{k-2}(f), & \\
		\label{dofV4}
		&\ \bullet \mathbf{D}^4_{\boldsymbol{V}}: \text{the face moments  } \frac{1}{|f|}\int_{f} \boldsymbol{v} \cdot \boldsymbol{n}_{f} p_{r-1} \d f, \forall p_{r-1} \in P_{r-1}(f), &\\
		\label{dofV5}
		&\ \bullet \mathbf{D}^5_{\boldsymbol{V}}: \text{the volume moments } \frac{1}{|K|}\int_K \boldsymbol{v} \cdot \boldsymbol{p}_{k-3} \d K,  \forall \boldsymbol{p}_{k-3} \in \nabla P_{k-2}(K),&\\
		\label{dofV6}
		&\ \bullet \mathbf{D}^6_{\boldsymbol{V}}:\text{the volume moments } 
		\frac{1}{|K|}\int_K \boldsymbol{v} \cdot (\boldsymbol{x}_K\wedge \boldsymbol{p}_{r-2}) \d K,  \forall \boldsymbol{p}_{r-2} \in \boldsymbol{P}_{r-2}(K),
	\end{flalign}
	are unisolvent in $\Vrk$. 
\end{proposition}

\begin{proof} 
	Note that the serendipity space $\widetilde{\boldsymbol{V}}_{0,k+1}(K)$ is initially equipped with the extra degrees of freedom
	\begin{equation*}
		\widetilde{\mathbf{D}}^6_{\boldsymbol{V}}: \text{the volume moments }  \int_{K} \boldsymbol{v} \cdot (\boldsymbol{x}_K \wedge \boldsymbol{p}_{0}) \d K,  \forall \boldsymbol{p}_0 \in \boldsymbol{P}_0(K).
	\end{equation*}
	However, by the very definition of \eqref{Vr0}, these moments are identically zero. Therefore, for simplicity, we let $\mathbf{D}^6_{\boldsymbol{V}}$ (for $r=1$) denote the vanishing form of $\widetilde{\mathbf{D}}^6_{\boldsymbol{V}}$.
	By a standard argument for reducing enlarged virtual element spaces to their restricted counterparts \cite{Ahmad2013}, we only need to establish the unisolvence of the degrees of freedom \eqref{dofV1}–\eqref{dofV6} along with the moments
	\begin{equation*}
		\widetilde{\mathbf{D}}^5_{\boldsymbol{V}}: \text{the volume moments } \int_{K} \boldsymbol{v} \cdot \widehat{\boldsymbol{p}}_{k-1} \d K,  \forall \widehat{\boldsymbol{p}}_{k-1} \in \nabla P_{k}(K)/\nabla P_{k-2}(K).
	\end{equation*}
	in the enlarged space $\widetilde{\boldsymbol{V}}_{r-1,k+1}(K)$.
	\par 
	We begin the proof by counting the number of degrees of freedom:
	\begin{gather*}
		\# \mathbf{D}^1_{\boldsymbol{V}}= N_v,\quad \# \mathbf{D}^2_{\boldsymbol{V}}=(k-1)N_e,\quad \# \mathbf{D}^3_{\boldsymbol{V}}=\text{dim}(P_{k-2}(f))N_f\\
		\# \mathbf{D}^4_{\boldsymbol{V}}=\text{dim}(P_{r-1}(f))N_f,\quad \# \mathbf{D}^5_{\boldsymbol{V}}+\# \widetilde{\mathbf{D}}^5_{\boldsymbol{V}}=(\text{dim}(P_{k}(K))-1),\\
		\mathbf{D}^6_{\boldsymbol{V}}=3\text{dim}(P_{r-1}(K))-\dim(P_r(K))+1.
	\end{gather*}
	The total number is then seen to match the dimension of $\widetilde{\boldsymbol{V}}_{r-1, k+1}(K)$ given in \eqref{dimEnlargedV}.  Moreover, we observe that the dimension in \eqref{DimVrk} is obtained by subtracting from \eqref{dimEnlargedV} the number of additional moments $\# \widetilde{\mathbf{D}}^5_{\boldsymbol{V}}$, which is given by $\dim(P_{k}(K)) - \dim(P_m(K))$. It remains to show that if all degrees of freedom \eqref{dofV1}–\eqref{dofV6} vanish for any $\boldsymbol{v} \in \widetilde{\boldsymbol{V}}_{r-1, k+1}(K)$, then $\boldsymbol{v}$ is identically zero. 
	\par 
	By definition, $(\operatorname{div} \boldsymbol{v})|_{\partial K} \in \mathbb{B}_k(\partial K)$. The vanishing of $\mathbf{D}^i_{\boldsymbol{V}}(\boldsymbol{v})$ for $i=1,2,3$, together with the unisolvence of $\mathbb{B}_k(\partial K)$, then allows us to conclude that
	\begin{equation}\label{divv0}
		\div \v =0 \quad \text{on } \partial K.
	\end{equation}
	Note that $\boldsymbol{v} \cdot \boldsymbol{n}_f \in P_{r-1}(f) $ on each face $f\in \partial K$,  $\mathbf{D}^4_{\boldsymbol{V}}(\v)=0$ implies 
	\begin{equation}\label{vn0}
		\boldsymbol{v} \cdot \boldsymbol{n}_{\partial K}=0 \quad \text{on } \partial K.
	\end{equation}
	On the element $K$, the homogeneous boundary conditions \eqref{divv0} and \eqref{vn0} lead to 
	\begin{align*}
		|\div \v|_{1,K}^2&=\int_K \nabla\div \v \cdot\nabla\div \v \d K\\
		&=-\int_K \div \v \Delta\div \v \d K +\int_{\partial K} \div \v \frac{\partial \div \v }{\partial \boldsymbol{n}_{\partial K}} \d s\\
		& =\int_K \v \cdot (\nabla\div)^2 \v \d K -\int_{\partial K} \v \cdot\boldsymbol{n}_{\partial K} \Delta\div \v \d s \\
		& =\int_K \v \cdot (\nabla\div)^2 \v \d K
	\end{align*}
	Combined  with the definition of \eqref{EnlargedV} and  $\mathbf{D}^{5}_{\boldsymbol{V}}(\boldsymbol{v})=\widetilde{\mathbf{D}}^5_{\boldsymbol{V}}(\boldsymbol{v})=0$, there exists $p_k\in P_k(K)$ such that 
	\begin{equation*}
		\begin{aligned}
			|\div \boldsymbol{v}|_{1,K}^2 = \int_K \v \cdot \nabla p_K \d K=0.
		\end{aligned}
	\end{equation*}
	Then, applying the Poincar\'{e} inequality under the homogeneous boundary condition \eqref{divv0} yields
	\begin{equation}\label{div0}
		\operatorname{div} \boldsymbol{v}=0 \quad \text{in }  K.
	\end{equation}
	From the exactness of \eqref{ContiComplex},  there exists a function $\boldsymbol{\phi} \in \boldsymbol{H}_0(\operatorname{curl};K)$ such that 
	\begin{equation*}
		\boldsymbol{v}=\operatorname{curl}\boldsymbol{\phi} \quad \text{in } K.
	\end{equation*} 
	By the definition of spaces $\boldsymbol{V}_{r-1,k+1}(K)$ and $\boldsymbol{\Sigma}_{0,r}(K)$, we have $\boldsymbol{\phi} \in \boldsymbol{\Sigma}_{0,r}(K)$ and $\mathbf{D}^1_{\boldsymbol{\Sigma}}(\boldsymbol{\phi})=\mathbf{D}^2_{\boldsymbol{\Sigma}}(\boldsymbol{\phi})=0$. Given $\mathbf{D}^6_{\boldsymbol{V}}(\boldsymbol{v})=0$,  we get $\mathbf{D}^3_{\boldsymbol{\Sigma}}(\boldsymbol{\phi})=\mathbf{D}^6_{\boldsymbol{V}}(\boldsymbol{v})=0$, which  implies  $\boldsymbol{\phi} =0$. 
	The proof is complete.
\end{proof}
\begin{remark}\label{PolyCompl}
	From the polynomial decomposition \eqref{Pkdc2}, we can easily find that
	\begin{equation*}
		\begin{aligned}
			\boldsymbol{P}_l(K) = \operatorname*{curl}\boldsymbol{P}_{l+1}(K)\oplus\boldsymbol{x}P_{l-1}(K) &\subset \operatorname{curl}\boldsymbol{P}_{r}(K)\oplus \boldsymbol{x}P_{l-1}(K)  \subset \boldsymbol{V}_{r-1, k+1}(K), \\
			P_k(K)&\subset\operatorname{div}\boldsymbol{V}_{r-1, k+1}(K),  
		\end{aligned}
	\end{equation*}
	where $l$=min$\{r-1, k+1\}$. Then taking $r=k$, $r=k+1$ and $r=k+2$, we have $\boldsymbol{P}_{k-1}(K)\subset\boldsymbol{V}_{k-1, k+1}(K)$, $\boldsymbol{P}_{k}(K)\subset\boldsymbol{V}_{k, k+1}(K)$ and  $\boldsymbol{P}_{k+1}(K)\subset\boldsymbol{V}_{k+1, k+1}(K)$, respectively. 
\end{remark}
\begin{proposition}\label{L2Prop}
	According to the degrees of freedom \eqref{dofV1}-\eqref{dofV6}, 
	the projections $\Pi_{k}^{\nabla,K}\div: \boldsymbol{V}_{r-1,k+1}(K)\to P_{k}(K)$ and $ \Pi_{l}^{0, K}: \boldsymbol{V}_{r-1, k+1}(K) \to \boldsymbol{P}_{l}(K)$ with $l=\min\{r-1,k-1\}$ are computable.
\end{proposition}
\begin{proof}
	For any $p_k \in P_k(K)$, it holds that 
	\begin{align*}
		(\nabla \Pi^{\nabla,K}_k \div \v , \nabla p_k)_K&=(\nabla\div \v ,\nabla p_k)_K\\
		&=-\int_K \div \v \,  \Delta p_k \d K + \int_{\partial K}\div \v \,\frac{\partial p_k}{\partial \boldsymbol{n}_{\partial K}} \d s \\
		&=\int_K \v \cdot \nabla\Delta p_k \d K -\int_{\partial K} \v\cdot \boldsymbol{n}_{\partial K} \Delta p_k \d s+  \int_{\partial K}\div \v \,\frac{\partial p_k}{\partial \boldsymbol{n}_{\partial K}} \d s.
	\end{align*}
	The first two terms are computable directly from \eqref{dofV4} and \eqref{dofV5}, respectively. For the last term, we can compute it based on the definition of $\mathbb{B}_k(\partial K)$, along with the degrees of freedom given in \eqref{dofV1}--\eqref{dofV3}.
	\par 
	Using the decomposition \eqref{Pkdc1}, for any given $\boldsymbol{q}_{l} \in \boldsymbol{P}_{l}(K)$, there exist three polynomials $\widehat{q}_{l+1}\in \widehat{P}_{l+1/m}(K)$, $q_{m}\in P_{m}(K)$ and $\boldsymbol{q}_{l-1}\in\boldsymbol{P}_{l-1}(K)$ with $m=\max\{0,k-2\}$, such that 
	\begin{equation*}
		\boldsymbol{q}_{l} = \nabla \widehat{q}_{l+1}+ \nabla q_m + \boldsymbol{x}_K\wedge \boldsymbol{q}_{l-1}.
	\end{equation*}
	From the definition of projection $\Pi_{l}^{0, K}$ and integration by parts, for any $\boldsymbol{v}\in \Vrk$, we have 
	\begin{equation}\label{ProjV}
		\begin{aligned}
			(\Pi_{l}^{0, K}\boldsymbol{v}, \boldsymbol{q}_{l})_K&=(\boldsymbol{v}, \boldsymbol{q}_{l})_K=(\boldsymbol{v}, \nabla \widehat{q}_{l+1})_K+(\boldsymbol{v}, \nabla q_m)_K+(\v,\boldsymbol{x}_K\wedge \boldsymbol{q}_{l-1})_K\\
			&=-\int_K \operatorname{div} \boldsymbol{v} \widehat{q}_{l+1} \d K+\int_{\partial K}\boldsymbol{v}\cdot \boldsymbol{n}_{\partial K}  \widehat{q}_{l+1} \d s+ \int_K \boldsymbol{v} \cdot \nabla q_m \d K + \int_K \boldsymbol{v} \cdot (\boldsymbol{x}_K\wedge\boldsymbol{q}_{l-1}) \d K\\
			&=-\int_K (\Pi^{\nabla, K}_{l+1} \operatorname{div} \boldsymbol{v} )\widehat{q}_{l+1} \d K +\sum_{f\in \partial K} \int_f \boldsymbol{v}\cdot \boldsymbol{n}_f  \widehat{q}_{l+1} \d f+\int_K \boldsymbol{v} \cdot \nabla q_m \d K \\
			&\quad + \int_K \boldsymbol{v} \cdot (\boldsymbol{x}_K\wedge\boldsymbol{q}_{l-1}) \d K.
		\end{aligned}
	\end{equation}
	According to the computable $\Pi_{k}^{\nabla,K}\div$ and  the degrees of freedom \eqref{dofV4}-\eqref{dofV6}, we can compute the terms on the right-hand side of \eqref{ProjV}.
\end{proof}
Gluing the local space $\boldsymbol{V}_{r-1, k+1}(K)$ over all elements $K$ in $\mathcal{T}_h$ produces the global space
\begin{equation}
	\boldsymbol{V}_{r-1, k+1}(\Omega):=\{ \boldsymbol{v} \in \boldsymbol{V}( \Omega): \boldsymbol{v}|_K \in \boldsymbol{V}_{r-1, k+1}(K), \forall K \in \mathcal{T}_h\}.
\end{equation}
We note that the global set of degrees of freedom, defined as the counterpart of \eqref{dofV1}–\eqref{dofV4}, guaranties the conforming property $\div \boldsymbol{V}_{r-1,k+1}(\Omega) \subset H^1(\Omega)$.

\subsection{The discrete complex}
\begin{theorem}\label{DisExactTheorem}
	The discrete complex \eqref{discomplex} is exact.
\end{theorem}
\begin{proof}
	The exactness of the first two components of the discrete complex is established in \cite{VEMforMaxwell,VEMforlowMaxwell}:
	\begin{equation*}
		\nabla U_{1}(\Omega) = \operatorname{ker}(\operatorname{curl})\cap \boldsymbol{ \Sigma}_{0,r}(\Omega) \quad \text{and}\quad \operatorname{curl} \boldsymbol{ \Sigma}_{0,r}(\Omega) = \operatorname{ker (div)} \cap \boldsymbol{ V}_{r-1,k+1}(K).
	\end{equation*}
	Then, it remains to prove the  exactness of the final component:
	\begin{equation}\label{SurDiv}
		\operatorname{div}\boldsymbol{V}_{r-1,k+1}(\Omega) = W_k(\Omega).
	\end{equation}
	Due to the compatibility with the div operator inherent in the definitions \eqref{WFinal} and \eqref{FinalSpace}, establishing the global exact sequence \eqref{SurDiv} requires only proving the exactness of the local enlarged spaces $\widetilde{\boldsymbol{V}}_{r-1,k+1}(K)$ and $\widetilde{\boldsymbol{W}}_k(K)$. The global result is subsequently obtained by patching these local constructions.
	\par 
	According to Remark \ref{EquWk}, any $w^K \in \widetilde{W}_k(K)$ is the solution of the Poisson equation \eqref{EquH1problem} with data $j \in \widehat{P}_{k/0}(K)$, $g_1 \in \mathbb{B}_k(\partial K)$, and the constraint $\int_K w^K \d K \in P_0(K)$.
	Similarly, any $\boldsymbol{u}^{K} \in \boldsymbol{V}_{r-1,k+1}(K)$ is the unique solution to the local quad-div problem \eqref{NonHomQuadDiv} on $K$, with data  $$\boldsymbol{f}^d \in \nabla P_k(K), \quad  \boldsymbol{f}^c \in \boldsymbol{P}_{r-2}(K)\cap\ker(\div), \quad (h, g_2) \in (P_{r-1}(\partial K),\mathbb{B}_k(\partial K)) \subset Y(\partial K).$$
	Therefore, by setting $\boldsymbol{f}^d = \nabla j$, $g_2 = g_1$, and  $\int_{\partial K} h  \d s=\int_K w^K \d K$, Theorem \ref{NonHomTheorem} implies that
	\begin{equation*}
		\div \boldsymbol{u}^{K} = w^{K},
	\end{equation*}
	from which it follows that
	\begin{equation*}
		\div\widetilde{\boldsymbol{V}}_{r-1,k+1}(K) = \widetilde{W}_k(K).
	\end{equation*}
\end{proof}
\par 
For $s>\frac{1}{2}$,  we introduce interpolation operators defined via the degrees of freedom of their respective virtual element spaces:
\begin{align}
	\widehat{J}_h: H^{s+1}(\Omega) \to U_1(\Omega),\\
	\boldsymbol{I}_h^e:\boldsymbol{H}^s(\operatorname{curl};\Omega)\to \boldsymbol{\Sigma}_{0,r}(\Omega),\\
	\boldsymbol{I}_h: \boldsymbol{V}^s(\Omega)\to \boldsymbol{V}_{r-1,k+1}(\Omega),\\
	J_h:W^s(\Omega)= H^{1+s}(\Omega) \to W_k(\Omega),
\end{align}
where $\boldsymbol{H}^s(\curl;\Omega):=\{\v \in \boldsymbol{H}^s(\Omega):\curl \v \in \boldsymbol{H}^s(\Omega)\}$ and $\boldsymbol{V}^s(\Omega):=\{ \v \in\boldsymbol{H}^s(\Omega):\div \v \in H^{s+1}(\Omega)\}.$
The continuous embedding $H^{s+1}(\Omega) \hookrightarrow C^{0,s-\frac{1}{2}}(\Omega)$  and the trace theorem \eqref{Trace0} ensure that $\widehat{J}_h$, $\boldsymbol{I}_h^e$, $\boldsymbol{I}_h$, and $J_h$,  are well-defined.  
\begin{remark}
	It follows from the regularity result in Remark \ref{EquAndReg} that if $\boldsymbol{f} \in \boldsymbol{H}^{s-2}(\Omega) \cap \ker(\curl)$ with $s > \frac{1}{2}$, then the solution $\boldsymbol{u}$ to the quad-div problem \eqref{mixForm} possesses the regularity $\boldsymbol{u} \in \boldsymbol{V}^s(\Omega)$, and thus the interpolant $\boldsymbol{I}_h \boldsymbol{u}$ is well-defined.
\end{remark}
\begin{proposition}
	The last two rows of the following diagram are commutative.
	\begin{equation}\label{CommutDigram}
		\begin{tikzcd}
			\mathbb{R} \arrow[r, "\subset"] & H^1(\Omega) \arrow[r, "\nabla"] \arrow[d,"{\rotatebox{90}{$\subset$}}"] & \boldsymbol{H}(\operatorname{curl}; \Omega) \arrow[r, "\operatorname{curl}"] \arrow[d,"{\rotatebox{90}{$\subset$}}"] & \boldsymbol{V}(\Omega) \arrow[r, "\operatorname{div}"] \arrow[d,"{\rotatebox{90}{$\subset$}}"] & H^1(\Omega)\arrow[r] \arrow[d,"{\rotatebox{90}{$\subset$}}"] & 0 
			\\
			\mathbb{R} \arrow[r, "\subset"] & H^{s+1}(\Omega) \arrow[r, "\nabla"] \arrow[d, "\widehat{J}_h"] & \boldsymbol{H}^s(\curl;\Omega) \arrow[r, "\operatorname{curl}"] \arrow[d, "\boldsymbol{I}_h^e"] & \boldsymbol{V}^s( \Omega) \arrow[r, "\operatorname{div}"] \arrow[d, "\boldsymbol{I}_h"] & H^{s+1}(\Omega) \arrow[r] \arrow[d, "J_h"] & 0 
			\\
			\mathbb{R} \arrow[r, "\subset"] & U_1(\Omega) \arrow[r, "\nabla"] & \boldsymbol{\Sigma}_{0,r}(\Omega) \arrow[r, "\operatorname{curl}"] & \boldsymbol{V}_{r-1,k+1}(\Omega) \arrow[r, "\operatorname{div}"] & W_k(\Omega) \arrow[r] & 0 .
		\end{tikzcd}
	\end{equation}
\end{proposition}

\begin{proof}
	For any $q\in U^s(\Omega)$, by the definition of $\mathbf{D}^1_{U}$ and $\mathbf{D}^1_{\boldsymbol{\Sigma}}$, there holds
	\begin{align*}
		\int_e	\nabla \widehat{J}_h q \cdot \boldsymbol{t}_e\d e &= \widehat{J}_hq(v_2)-\widehat{J}_hq(v_1)
		=q(v_2)-q(v_1)\\
		&=\int_e \nabla q \cdot \boldsymbol{t}_e \d e
		= \int_e \boldsymbol{I}_h^e \nabla q \cdot \boldsymbol{t}_e \d e,
	\end{align*}
	where $v_1$ and $v_2$ are the endpoints of edge $e$ with $\overrightarrow{v_1v_2} = \boldsymbol{t}_e$. This implies that $\mathbf{D}^1_{\boldsymbol{\Sigma}}(\nabla \widehat{J}_h q)= \mathbf{D}^1_{\boldsymbol{\Sigma}}(\boldsymbol{I}_h^e \nabla q)$. Combined with the identities $\mathbf{D}^2_{\boldsymbol{\Sigma}}(\nabla \widehat{J}_h q)=\mathbf{D}^2_{\boldsymbol{\Sigma}}(\boldsymbol{I}_h^e \nabla q)=\mathbf{D}^3_{\boldsymbol{\Sigma}}(\nabla \widehat{J}_h q)=\mathbf{D}^3_{\boldsymbol{\Sigma}}(\boldsymbol{I}_h^e \nabla q)=0$ for $r\ge 2$, the unisolvence of the degrees of freedom for $\boldsymbol{\Sigma}_{0,r}(K)$ yields
	\begin{equation}\label{nablacom}
		\nabla \widehat{J}_h q	=\boldsymbol{I}_h^e \nabla q \quad \text{in } K.
	\end{equation}
	For any $\boldsymbol{\phi}\in \boldsymbol{H}^s(\curl;\Omega)$, the definitions of $\mathbf{D}^4_{\boldsymbol{V}}$, $\mathbf{D}^1_{\boldsymbol{\Sigma}}$, $\mathbf{D}^2_{\boldsymbol{\Sigma}}$, and the polynomial decomposition
	\begin{equation*}
		q_{r-1} = \widehat{q}_{r-1} + q_0, \quad \forall q_{r-1}\in P_{r-1}(K), \widehat{q}_{r-1}\in \widehat{P}_{r-1/0}(K)(K), q_0 \in P_0(K),
	\end{equation*}
	yield
	\begin{align*}
		(\operatorname{curl}\boldsymbol{I}^e_h \boldsymbol{ \phi}\cdot \boldsymbol{n}_f,q_{r-1})_f 
		&=(\operatorname{rot}(\boldsymbol{I}^e_h \boldsymbol{ \phi})_{\tau},\widehat{q}_{r-1})_f +(\operatorname{rot}(\boldsymbol{I}^e_h \boldsymbol{ \phi})_{\tau},q_{0})_f\\
		&=(\operatorname{rot} \boldsymbol{ \phi}_{\tau},\widehat{q}_{r-1})_f + \int_{\partial f} (\boldsymbol{I}^e_h \boldsymbol{ \phi})_{\tau}\cdot \boldsymbol{ t}_{\partial f} q_0\d s \\
		&=(\operatorname{rot} \boldsymbol{ \phi}_{\tau},\widehat{q}_{r-1})_f + \int_{\partial f}  \boldsymbol{ \phi}_{\tau}\cdot \boldsymbol{ t}_{\partial f}q_0 \d s \\
		&= (\operatorname{rot} \boldsymbol{ \phi}_{\tau}, q_{r-1})_f\\
		&=(\boldsymbol{I}_h\operatorname{curl} \boldsymbol{ \phi}\cdot \boldsymbol{n}_f,q_{r-1})_f,
	\end{align*}
	implying 
	\begin{equation*}
		(\operatorname{curl}\boldsymbol{I}^e_h\boldsymbol{\phi}-\boldsymbol{I}_h\operatorname{curl}\boldsymbol{\phi})|_{\partial K} = 0.
	\end{equation*}
	Furthermore, applying integration by parts yields $\mathbf{D}^5_{\boldsymbol{V}}(\operatorname{curl}\boldsymbol{I}^e_h \boldsymbol{\phi}-\boldsymbol{I}_h\operatorname{curl} \boldsymbol{\phi})=0$. 
	The definition of $\mathbf{D}_{\boldsymbol{\Sigma}}^3$ implies that $\mathbf{D}^6_{\boldsymbol{V}}(\operatorname{curl}\boldsymbol{I}^e_h \boldsymbol{ \phi}) =\mathbf{D}^6_{\boldsymbol{V}}(\boldsymbol{I}_h\operatorname{curl}\boldsymbol{ \phi})$. Then, all degrees of freedom \eqref{dofV1}--\eqref{dofV6} for the difference $\operatorname{curl}\boldsymbol{I}_h^e \boldsymbol{\phi}-\boldsymbol{I}_h\operatorname{curl}\boldsymbol{\phi}$ vanish, which yields 
	\begin{equation}\label{CurlCom}
		\operatorname{curl}\boldsymbol{I}_h^e \boldsymbol{\phi} = \boldsymbol{I}_h\operatorname{curl}\boldsymbol{\phi} \quad \text{in } K.
	\end{equation}
	For any $\boldsymbol{v} \in \boldsymbol{V}^s(\Omega)$, using the face degrees of freedom of the operators $\boldsymbol{I}_h$ and $J_h$, together with the unisolvence of $\mathbb{B}_k(\partial K)$, we obtain
	\begin{equation}\label{commut1}
		(\operatorname{div} \boldsymbol{I}_h\boldsymbol{v} -J_h\operatorname{div}\boldsymbol{v})|_{\partial K} =0.
	\end{equation}
	From the degrees of freedom $\mathbf{D}_{\boldsymbol{V}}^4$, $\mathbf{D}_{\boldsymbol{V}}^5$, and $\mathbf{D}_{W}^4$, it follows that
	\begin{equation}\label{commut2}
		\begin{aligned}
			(\operatorname{div}\boldsymbol{I}_h \boldsymbol{v}, q_m)_K&=  -(\boldsymbol{I}_h \boldsymbol{v}, \nabla q_m)_K + \int_{\partial K} \boldsymbol{I}_h\boldsymbol{v} \cdot \boldsymbol{n}_{\partial K} q_m \d s\\
			&=-(\boldsymbol{v}, \nabla q_m)_K +\int_{\partial K} \boldsymbol{v}\cdot \boldsymbol{n}_{\partial K} q_m \d s \\
			&=(\operatorname{div}\boldsymbol{v}, q_m)_K\\
			&=(J_h\operatorname{div}\boldsymbol{v}, q_m)_K, \quad \forall q_m\in P_m(K),  m=\operatorname{max}\{0, k-2\}.
		\end{aligned}
	\end{equation}
	Recalling the unisolvence of $W_k(K)$, and combining \eqref{commut1} and \eqref{commut2},  we obtain
	\begin{equation}\label{commut}
		\operatorname{div}\boldsymbol{I}_h \boldsymbol{v}= J_h \operatorname{div}\boldsymbol{v}\quad \text{in } K.
	\end{equation}
	Thus, \eqref{nablacom}, \eqref{CurlCom}, and \eqref{commut} give the desired commutativity.
\end{proof}
\begin{remark}
	The $\boldsymbol{H}(\operatorname{grad-div})$-conforming virtual element spaces proposed here generalize the three finite element families in \cite{ZZMQuadDiv} (with $r = k$, $k+1$, $k+2$) to polyhedral meshes. For the lowest-order case $r=k=1$, our virtual element degrees of freedom, defined for general polyhedra (see Figure \ref{fig:dofgraddiv}), recover exactly those of the discrete finite element complex in \cite{ZZMQuadDiv} on simplicial and cuboid meshes.
\end{remark}
\begin{figure}
	\centering
	\includegraphics[width=1.0\linewidth, trim=0 3cm 0 3cm, clip]{GradDivDofComplex.png}
	\caption{The lowest-order ($r = k = 1$) virtual element complex \eqref{discomplex} on a polyhedral element.}
	\label{fig:dofgraddiv}
\end{figure}
\subsection{Interpolation and stability results}
As noted in Remark \ref{PolyCompl}, the polynomial compatibility depends on the choice of $r$.
For ease of discussion, we focus on the case of  $r = k$ in the following, as it minimizes the degrees of freedom. Similar arguments can be used for the other cases. In this subsection, we give the interpolation error estimates and the stability result for the local space $\boldsymbol{V}_{k-1,k+1}(K)$.
\par 
To establish these results, we impose the following additional geometric assumption on mesh faces:
\begin{itemize}
	\item  Each face $f$ is convex, and there exists a constant $\epsilon>0$ such that every interior angle $\theta$ of $f$ satisfies $\epsilon \le \theta \le \pi -\epsilon $.
\end{itemize}
This convexity condition permits a shape-regular simplicial subdivision of each face or element, which guaranties the validity of certain polynomial inverse estimates; see \cite[Remark 1]{VEMforGener} for details.
We list several such estimates. First, the standard polynomial inverse estimate \cite[Section 3.6]{Verfurth2013} gives
\begin{equation}\label{PolyInv}
	\|p_k\|_{\partial K} \lesssim h_K^{-\frac{1}{2}}\|p_k\|_{-\frac{1}{2},\partial K}, \quad \forall p_k \in P_k(\partial K),
\end{equation}
Second, for the $H^1$-conforming virtual element space $\widehat{\mathbb{B}}_k(f)$, we have \cite[Theorem 3.6]{ChenLong2018}:
\begin{equation}\label{BkInv}
	|v|_{1,f}\lesssim h_f^{-1} \|v\|_f, \quad v \in \widehat{\mathbb{B}}_k(f).
\end{equation}
Finally, for any $\boldsymbol{v}_h \in \boldsymbol{V}_{k-1,k+1}(K)$ with
$k\ge 2$ ($\curl \boldsymbol{v}_h \in \boldsymbol{P}_{k-2}(K) \cap \ker(\div)$), the following inverse estimate for the curl operator is valid \cite[Lemma 4.1]{VEMforGener}:
\begin{equation}\label{CurlInverse}
	\|\curl \boldsymbol{v}_h\|\lesssim h_K^{-1}\|\v_h\|_K.
\end{equation} 
In addition, we recall two trace estimates. For any element or face $G$
\cite[Theorem A.20]{Schwab1998},
\begin{equation}\label{H1trace}
	\|v\|_{\partial G} \lesssim h_G^{-\frac{1}{2}}\|v\|_G + h_G^{\frac{1}{2}}|v|_{1,G},\quad \forall v\in H^1(G).
\end{equation}
For any element $K$, a scaled version of the trace estimate \eqref{HDivTrace} gives 
\begin{align}
	\label{HdivTrace}
	\|\boldsymbol{v}\cdot \boldsymbol{n}_{\partial K}\|_{-\frac{1}{2},\partial K}&\lesssim \|\boldsymbol{v}\|_K+ h_K\|\div \v\|_K,\quad\forall \v \in \boldsymbol{H}(\operatorname{div};K).
\end{align}
\par 
We slightly modify the  polynomial degrees in $\boldsymbol{H}(\operatorname{div})$-conforming virtual element space introduced by  \cite{VEMforGener} to obtain for $k\ge 2$:
\begin{align*}
	\boldsymbol{V}^f_{k-1, k+1}(K)=\left\{ \boldsymbol{v} \in \boldsymbol{L}^2(K): \operatorname{div}\boldsymbol{v}\in P_{k}(K), \operatorname{curl}\boldsymbol{v}\in \boldsymbol{P}_{k-2}(K),\right. \\
	\left.(\boldsymbol{v}\cdot \boldsymbol{n}_{\partial K})|_{\partial K} \in P_{k-1}(\partial K)\right\}.
\end{align*} 
It is endowed with the following degrees of freedom 
\begin{flalign}\label{dofVf1}
	&\ \bullet  \mathbf{D}^1_{\boldsymbol{V}^f}: \text{the face moments  } \frac{1}{|f|}\int_{f} \boldsymbol{v} \cdot \boldsymbol{n}_{f} p_{k-1} \d f,  \forall p_{k-1} \in P_{k-1}(f), &\\
	&\ \bullet \mathbf{D}^2_{\boldsymbol{V}^f}: \text{the volume moments } \frac{1}{|K|}\int_K \boldsymbol{v} \cdot \boldsymbol{p}_{k-1} \d K,  \forall \boldsymbol{p}_{k-1} \in \nabla P_{k}(K),&\\
	&\ \bullet \mathbf{D}^3_{\boldsymbol{V}^f}:\text{the volume moments } \label{dofVf3}
	\frac{1}{|K|}\int_K \boldsymbol{v} \cdot (\boldsymbol{x}_K\wedge \boldsymbol{p}_{k-2}) \d K, \forall \boldsymbol{p}_{k-2} \in \boldsymbol{P}_{k-2}(K).
\end{flalign}
For $k=1$, the serendipity space are given in \cite{VEMforlowMaxwell}:
\begin{align*}
	\boldsymbol{V}^f_{0, 2}(K)=\left\{ \boldsymbol{v} \in \boldsymbol{L}^2(K): \operatorname{div}\boldsymbol{v}\in P_{1}(K), \operatorname{curl}\boldsymbol{v}\in \boldsymbol{P}_{0}(K),
	(\boldsymbol{v}\cdot \boldsymbol{n}_{\partial K})|_{\partial K} \in P_{0}(\partial K) ,\right. \\ \left.\int_K\boldsymbol{v}\cdot( \boldsymbol{x}_K\wedge \boldsymbol{p}_0) \d K=0, \forall \boldsymbol{p}_0 \in \boldsymbol{P}_{0}(K)\right\},
\end{align*} 
equipped with the face moments \eqref{dofVf1} with $k=1$.
As shown in \cite[Lemma 4.1]{VEMforGener}, an auxiliary bound holds for the space $\boldsymbol{V}^f_{k-1, k+1}(K)$. 
We now extend this estimate to the space $\boldsymbol{V}_{k-1,k+1}(K)$ and their direct sum, following an analogous argument.
\begin{lemma}\label{lowbound}
	For each $\boldsymbol{v}_h \in  \boldsymbol{V}_{k-1, k+1}(K)$, $ \boldsymbol{V}^f_{k-1, k+1}(K)$ or their sum space, we have
	\begin{equation}\label{Stabvh}
		\|\boldsymbol{v}_h\|_{K} \lesssim h_K\|\operatorname{div}\boldsymbol{v}_h\|_{K}+h_K^{\frac{1}{2}}\|\boldsymbol{v}_h\cdot \boldsymbol{n}_{\partial K}\|_{\partial K}+\sup _{\boldsymbol{p}_{k-2} \in\boldsymbol{P}_{k-2}(K)} \frac{\int_{K} \boldsymbol{v}_{h} \cdot (\boldsymbol{x}_{K} \wedge \boldsymbol{p}_{k-2})\d K}{\left\|\boldsymbol{x}_{K} \wedge \boldsymbol{p}_{k-2}\right\|_{K}}.
	\end{equation}
\end{lemma}
\begin{proof}
	Consider the following Helmholtz decomposition:
	\begin{equation}\label{Decvh}
		\boldsymbol{v}_h = \curl\boldsymbol{\rho}+\nabla\phi,
	\end{equation}
	where $\phi \in H^1(K)/\mathbb{R}$ satisfies the Poisson equation
	\begin{equation*}
		\Delta\phi = \div \v_h\text{ in } K,\quad \nabla\phi\cdot \boldsymbol{n}_{\partial K} = \boldsymbol{v}_h\cdot \boldsymbol{n}_{\partial K} \text{ on } \partial K
	\end{equation*}
	and $\boldsymbol{\rho}\in \boldsymbol{H}(\curl;K)$ satisfies weakly
	\begin{equation*}
		\begin{aligned}
			\curl\curl \boldsymbol{\rho}&=\curl \boldsymbol{v}_h,\quad \div\boldsymbol{\rho}=0 \text{ in } K, \quad \boldsymbol{\rho}\wedge\boldsymbol{n}_{\partial K}=0 \text{ on } \partial K.
		\end{aligned}
	\end{equation*}
	The well-posedness of both subproblems ensures the validity of this decomposition, and we have the orthogonal relations:
	\begin{equation}\label{DecElement}
		(\curl\boldsymbol{\rho},\nabla\phi)_K=0,\quad \|\boldsymbol{v}\|^2_K=\|\curl \boldsymbol{\rho}\|_K^2+\|\nabla\phi\|_K^2.
	\end{equation}
	Using the orthogonal decomposition \eqref{Decvh}, integration by parts, the trace estimate \eqref{H1trace}, together with the Poincaré inequality, we obtain
	\begin{equation}
		\begin{aligned}
			\|\nabla \phi\|_K^2 =(\nabla\phi , \boldsymbol{v}_h)_K&=-\int_K \div \v _h \phi \d K+ \int_{\partial K} \boldsymbol{v}_h \cdot \boldsymbol{n}_{\partial K} \phi \d s\\
			&\le \|\div\v_h\|_K\|\phi\|_K+\|\boldsymbol{v}_h \cdot \boldsymbol{n}_{\partial K}\|_{\partial K}\|\phi\|_{\partial K} \\
			&\lesssim(h_K \|\div\v_h\|_K+h_K^{\frac{1}{2}}\|\boldsymbol{v}_h\cdot \boldsymbol{n}_{\partial K}\|_{\partial K} )\|\nabla \phi\|_K
		\end{aligned}
	\end{equation}
	For $k\ge 2$, the fact that $\curl \v_h\in\boldsymbol{P}_{k-2}(K)\cap\ker(\div)$ together with
	\eqref{VectPolyIden} yields a polynomial $\boldsymbol{p}_{k-2}\in\boldsymbol{P}_{k-2}(K)$ satisfying $\curl \boldsymbol{v}_h = \curl (\boldsymbol{x}_{K} \wedge \boldsymbol{p}_{k-2})$. This in turn yields the bound
	\begin{equation}\label{xwedgeConti}
		\| \boldsymbol{x}_{K} \wedge \boldsymbol{p}_{k-2}\|_K \le h_K \|\curl \v_h\|_K.
	\end{equation}
	Employing again the orthogonal decomposition \eqref{Decvh}, integration by parts, the estimates \eqref{xwedgeConti} and \eqref{CurlInverse}, we derive
	\begin{equation}\label{curlrhoEsti}
		\begin{aligned}
			\|\curl \boldsymbol{\rho}\|_K^2&= \int_K \boldsymbol{\rho} \curl \curl \boldsymbol{\rho}\d K =\int_K \boldsymbol{\rho} \curl \v_h \d K\\
			&= \int_K \boldsymbol{\rho} \cdot \curl (\boldsymbol{x}_K\wedge \boldsymbol{p}_{k-2})\d K
			=\int_K (\boldsymbol{v}_h -\nabla \phi)\cdot (\boldsymbol{x}_K\wedge \boldsymbol{p}_{k-2})\d K\\
			&\le \left( \|\nabla\phi\|_K+\sup _{\boldsymbol{p}_{k-2} \in\boldsymbol{P}_{k-2}(K)} \frac{\int_{K} \boldsymbol{v}_{h} \cdot (\boldsymbol{x}_{K} \wedge \boldsymbol{p}_{k-2})\d K}{\left\|\boldsymbol{x}_{K} \wedge \boldsymbol{p}_{k-2}\right\|_{K}}\right)\|\boldsymbol{x}_K\wedge \boldsymbol{p}_{k-2}\|_K\\
			& \lesssim\left( \|\nabla\phi\|_K+\sup _{\boldsymbol{p}_{k-2} \in\boldsymbol{P}_{k-2}(K)} \frac{\int_{K} \boldsymbol{v}_{h} \cdot (\boldsymbol{x}_{K} \wedge \boldsymbol{p}_{k-2})\d K}{\left\|\boldsymbol{x}_{K} \wedge \boldsymbol{p}_{k-2}\right\|_{K}}\right)\|\v_h\|_K.
		\end{aligned}
	\end{equation}
	For the case $k=1$, we have $\curl \v_h \in \boldsymbol{P}_0(K)$; however, owing to the structure of the serendipity spaces, the supremum term in \eqref{curlrhoEsti} vanishes.
	Combining \eqref{DecElement}--\eqref{curlrhoEsti}, we obtain \eqref{Stabvh}.
\end{proof}
\par 
We proceed to introduce 
an auxiliary interpolation operator $\boldsymbol{I}_h^f: \boldsymbol{H}^s(K)\cap \boldsymbol{H}(\operatorname*{div}; K) \to \boldsymbol{V}^f_{k-1,k+1}(K)$  with $s >\frac{1}{2}$  based on the degrees of freedom \eqref{dofVf1}-\eqref{dofVf3},
who owns the  following interpolation error estimates  \cite[Theorem 4.2]{VEMforGener}.
\begin{lemma}\label{IntEstf}
	If $\boldsymbol{v}\in \boldsymbol{H}^s(K) $ and $\operatorname{div}\boldsymbol{v}\in H^{l}(K), \frac{1}{2}<s\le k $ and $  0 \le l \le k+1$,  then we have 
	\begin{align}\label{Ihf}
		\|\boldsymbol{v}-\boldsymbol{I}_h^f\boldsymbol{v}\|_K &\lesssim h_K^s|\boldsymbol{v}|_{s, K} +h_K\|\div \v\|_K,  \\
		\label{divIhf}
		\|\operatorname{div}(\boldsymbol{v}-\boldsymbol{I}_h^f\boldsymbol{v})\|_{K}&\lesssim h_K^l|\operatorname{div}\boldsymbol{v}|_{l, K}.
	\end{align}
	The second term on the right-hand side of \eqref{Ihf} can be neglected if $s\ge 1$.
\end{lemma}
\noindent
We also introduce the interpolation error estimates for $J_h$ as follows, see \cite[Theorem 4.3]{Vem3dH1Error}.
\begin{lemma}
	For every $q \in H^1_0(\Omega) \cap H^{s}(\Omega)$ with $\frac{3}{2} < s \le k+1$, it holds that
	\begin{equation}\label{interpolation_J}
		\|q-J_hq\|_K+h_K|q-J_hq|_{1, K} \lesssim h_K^s|q|_{s, K}.
	\end{equation}
\end{lemma}

Based on the above preparation, we give  the interpolation error estimates for $\boldsymbol{I}_h$.
\begin{theorem}\label{InEst}
	If $\boldsymbol{v} \in \boldsymbol{H}^s(\Omega)$ and $\operatorname{div}\boldsymbol{v} \in H^{s+1}(\Omega)$ with $\frac{1}{2} < s \le k $, there hold
	\begin{align}
		\label{Intpolation1}
		\|\boldsymbol{v}-\boldsymbol{I}_h\boldsymbol{v}\|_K&\lesssim h^s_K(|\boldsymbol{v}|_{s, K}+h^2_K|\operatorname{div}\boldsymbol{v}|_{s+1, K})+h_K\|\div \v\|_K, \\ 
		\label{Intpolation2}
		\|\operatorname{div}(\boldsymbol{v}-\boldsymbol{I}_h\boldsymbol{v})\|_K&\lesssim h_K^{s+1}|\operatorname{div}\boldsymbol{v}|_{s+1, K},\\
		\label{Intpolation3}
		|\operatorname{div}(\boldsymbol{v}-\boldsymbol{I}_h\boldsymbol{v})|_{1, K}&\lesssim h^s_K|\operatorname{div}\boldsymbol{v}|_{s+1, K}.
	\end{align}
	The last term on the right-hand side of \eqref{Intpolation1} can be neglected if $s\ge 1$.
\end{theorem}
\begin{proof} According to the commutative property \eqref{commut} between $\boldsymbol{I}_h $ and $J_h$ and the interpolation estimate \eqref{interpolation_J}, the results \eqref{Intpolation2} and \eqref{Intpolation3} can be easily obtained.  
	Considering the error $\boldsymbol{I}_h^f \boldsymbol{v}-\boldsymbol{I}_h\boldsymbol{v}$,
	from the properties of $\boldsymbol{I}_h^f $ and $\boldsymbol{I}_h$, we get
	\begin{align*}
		(\boldsymbol{I}^f_h\v - \boldsymbol{I}_h\v)\cdot \boldsymbol{n}_{\partial K} &= 0 \quad \text{on } \partial K, \\
		\int_K (\boldsymbol{I}^f_h\v - \boldsymbol{I}_h\v)\cdot (\boldsymbol{x}_K\wedge \boldsymbol{p}_{k-2})\d K &= 0, \quad \forall \boldsymbol{p}_{k-2} \in \boldsymbol{P}_{k-2}(K),
	\end{align*} 
	which, together with Lemma \ref{lowbound}, \eqref{divIhf} and \eqref{Intpolation2},  yields
	\begin{equation}\label{I-If}
		\begin{aligned}
			\|\boldsymbol{I}_h^f\boldsymbol{v}-\boldsymbol{I}_h\boldsymbol{v}\|_K &\lesssim h_K\|\operatorname{div}(\boldsymbol{I}_h^f\boldsymbol{v}-\boldsymbol{I}_h\boldsymbol{v})\|_{K}+h_K^{\frac{1}{2}}\|(\boldsymbol{I}_h^f\boldsymbol{v}-\boldsymbol{I}_h\boldsymbol{v})\cdot \boldsymbol{n}_{\partial K}\|_{\partial K}\\
			&\quad+\sup _{\boldsymbol{p}_{k-2} \in \boldsymbol{P}_{k-2}(K)} \frac{\int_{K} (\boldsymbol{I}_h^f\boldsymbol{v}-\boldsymbol{I}_h\boldsymbol{v}) \cdot (\boldsymbol{x}_{K} \wedge \boldsymbol{p}_{k-2})\d K}{\left\|\boldsymbol{x}_{K} \wedge \boldsymbol{p}_{k-2}\right\|_{K}}\\
			&\le h_K\|\operatorname{div}(\boldsymbol{v}-\boldsymbol{I}_h^f\boldsymbol{v})\|_K+h_K\|\operatorname{div}(\boldsymbol{v}-\boldsymbol{I}_h\boldsymbol{v})\|_K\\
			&\lesssim h_K^{s+2} |\operatorname{div}\boldsymbol{v}|_{s+1,K}.
		\end{aligned}
	\end{equation}
	Applying the triangle inequality, \eqref{Ihf} and \eqref{I-If}, we have 
	\begin{equation*}
		\begin{aligned}
			\|\boldsymbol{v}-\boldsymbol{I}_h\boldsymbol{v}\|_K&\le \|\boldsymbol{v}-\boldsymbol{I}_h^f\boldsymbol{v}\|_K+\|\boldsymbol{I}_h^f\boldsymbol{v}-\boldsymbol{I}_h\boldsymbol{v}\|_K\\
			&\lesssim h^s_K(|\boldsymbol{v}|_{s, K}+h^2_K|\operatorname{div}\boldsymbol{v}|_{s+1, K})+h_K\|\div \v\|_K.
		\end{aligned}
	\end{equation*}
	The proof is complete. 
\end{proof}
\par
Now we investigate the stability of bilinear form 
\begin{equation*}
	b_h^K(\cdot, \cdot ) :\boldsymbol{V}_{k-1,k+1}(K) \times \boldsymbol{V}_{k-1,k+1}(K) \to \mathbb{R}
\end{equation*}
defined by
\begin{equation}\label{bhk}
	b_h^K(\boldsymbol{v}_h, \boldsymbol{w}_h):=(\Pi_{k-1}^{0,K}\boldsymbol{v}_h, \Pi_{k-1}^{0,K}\boldsymbol{w}_h)_K+S^K((I-\Pi^{0,K}_{k-1})\boldsymbol{v}_h, (I-\Pi_{k-1}^{0,K})\boldsymbol{w}_h),
\end{equation}
where   
\begin{equation}\label{SK}
	\begin{aligned}
		S^K(\boldsymbol{\xi}_h, \boldsymbol{\eta}_h):=h_K^2(\Pi_{k}^{0, K}\operatorname{div}\boldsymbol{\xi}_h, \Pi_{k}^{0, K}\operatorname{div}\boldsymbol{\eta}_h)_{ K}+\sum_{f\in \partial K}\left[h_f^{3}(\Pi^{0, f}_k \operatorname{div}\boldsymbol{\xi}_h, \Pi^{0, f}_k \operatorname{div}\boldsymbol{\eta}_h)_{f} \right.\\
		\left.+h_f^{4}(\operatorname{div}\boldsymbol{\xi}_h, \operatorname{div}\boldsymbol{\eta}_h)_{ \partial f} 
		+h_f(\boldsymbol{\xi}_h\cdot \boldsymbol{n}_f, \boldsymbol{\eta}_h\cdot \boldsymbol{n}_f)_{ f}\right].
	\end{aligned}
\end{equation}
Since $\operatorname{div} \boldsymbol{v}_h \in W_k(K)$ for all $\boldsymbol{v}_h \in \boldsymbol{V}_{k-1,k+1}(K)$, Remark \ref{remark2} implies that the first three terms of \eqref{SK} are computable. The last term, meanwhile, is handled via \eqref{dofV5}.
\begin{lemma}\label{refLema4}
	There exist two positive constants $\alpha_*$ and $\alpha^* $ independent of $h_K$ such that  
	\begin{equation}\label{stab3}
		\begin{aligned}
			\alpha_*\|\boldsymbol{v}_h\|^2_{K}\le S^K(\boldsymbol{v}_h, \boldsymbol{v}_h)\le \alpha^* (\|\boldsymbol{v}_h\|^2_{K}+ h_K^2\|\operatorname*{div}\boldsymbol{v}_h\|_{K}^2+h_K^4|\div\v_h|_{1,K}^2),\\
			\quad  \forall \boldsymbol{v}_h \in \boldsymbol{V}_{k-1, k+1}(K)\cap \ker(\Pi^{0, K}_{k-1}).
		\end{aligned}
	\end{equation}
\end{lemma}
\begin{proof}
	For any $\boldsymbol{v}_h \in \boldsymbol{V}_{k-1, k+1}(K)\cap \ker(\Pi^{0, K}_{k-1})$, using Lemma \ref{lowbound} leads to
	\begin{equation}\label{aux3}
		\|\boldsymbol{v}_h\|_{K} \lesssim h_K\|\operatorname{div}\boldsymbol{v}_h\|_{K}+h_K^{\frac{1}{2}}\|\boldsymbol{v}_h\cdot \boldsymbol{n}_{\partial K}\|_{\partial K}.
	\end{equation}
	Due to the fact that $ \operatorname{div}\boldsymbol{v}_h \in W_k(K)$ and the definition of $S_n^K$ in Remark \ref{remark2}, it holds
	\begin{equation}\label{aux4}
		\begin{aligned}
			|\operatorname{div}\boldsymbol{v}_h|^2_{1, K}&\lesssim S_n^K(\div\boldsymbol{v}_h,\div\v_h)\\
			&=h_K^{-2}\|\Pi^{0, K}_k \div\v_h\|^2_K +\sum_{f\in\partial K} (h_f^{-1}\|\Pi^{0, f}_k\div\v_h\|^2_f +\|\div\v_h\|^2_{ \partial f}).
		\end{aligned}
	\end{equation}
	For the $L^2$-norm estimate, we obtain
	\begin{align*}
		\|\div\v_h\|_{K}&\le \|\div\v_h-\Pi^{0, K}_{k}\div\v_h\|_K+\|\Pi^{0, K}_{k}\div\v_h\|_K\\
		&\lesssim h_K|\div\v_h|_{1, K}+\|\Pi^{0, K}_{k}\div\v_h\|_K,
	\end{align*}
	which, together with \eqref{aux3} and \eqref{aux4}, yields 
	\begin{align*}
		\|\boldsymbol{v}_h\|_K&\lesssim h_K\|\Pi^{0, K}_k \operatorname{div}\boldsymbol{v}_h\|_K\\
		&\quad+\sum_{f\in \partial K}\left[h_f^{\frac{3}{2}} \|\Pi^{0, f}_k \operatorname{div}\boldsymbol{v}_h\|_f +h_f^2 \|\operatorname{div}\boldsymbol{v}_h\|_{\partial f}+h_f^{\frac{1}{2}}\|\boldsymbol{v}_h\cdot \boldsymbol{n}_{f}\|_{f}\right].
	\end{align*}
	This implies the lower bound in \eqref{stab3}.
	\par 
	Next we  estimate the four terms on the right-hand side of \eqref{SK} for the upper bound.
	Since the stability of projection $\Pi_{k}^{0,K}$, it holds
	\begin{equation}\label{PiStab}
		\|\Pi^{0,K}_{k}\div\v_h\|_K\lesssim \|\div\v_h\|_K.
	\end{equation}
	According to the stability of the projection $\Pi_{k}^{0,f}$ and  the trace inequality \eqref{H1trace}, we get  
	\begin{equation}\label{aux5}
		\begin{aligned}
			\sum_{f\in\partial K}h_f^{\frac{3}{2}}\|\Pi^{0, f}_k\operatorname{div} \v_h \|_f\lesssim h_K^{\frac{3}{2}}\|\operatorname{div}\v_h\|_{\partial K}
			\lesssim h_K\|\operatorname{div}\v_h\|_{ K}+h_K^2|\div\vh|_{1, K}.
		\end{aligned}
	\end{equation}
	Using the trace inequality \eqref{H1trace}, inverse inequality in $\widehat{\mathbb{B}}_k(f)$ \eqref{BkInv} obtains 
	\begin{equation}
		\begin{aligned}
			\sum_{f\in\partial K}h_f^2\|\div\vh\|_{\partial f} &\lesssim \sum_{f\in \partial K}h_f^2 (h_f^{-\frac{1}{2}}\|\div\vh\|_{ f}+h_f^{\frac{1}{2}}|\div\vh|_{1, f}) 
			\\ &\lesssim h_K^{\frac{3}{2}}\|\operatorname{div}\v_h\|_{\partial K} \lesssim h_K\|\div\vh\|_{K}+h_K^2|\div\vh|_{1, K}.
		\end{aligned}
	\end{equation}
	It follows from the polynomial inverse estimate \eqref{PolyInv} and the trace inequality \eqref{HdivTrace} that
	\begin{equation}\label{PolyIn}
		\sum_{f\in \partial K} h_f^{\frac{1}{2}}\|\vh\cdot \nf\|_f\lesssim \|\vh\cdot\nf\|_{-\frac{1}{2}, \partial K}\lesssim \|\vh\|_{ K}+h_K\|\div\vh\|_{ K}.
	\end{equation}
	Thus, we conclude from \eqref{PiStab}--\eqref{PolyIn} that the upper bound in \eqref{stab3} holds.
\end{proof}
\section{Discretization}
\subsection{The discrete bilinear forms}
In this subsection, we present three  bilinear forms  to discretize the  continuous problem \eqref{mixForm}.
For any $\boldsymbol{v}_h,\boldsymbol{w}_h \in \boldsymbol{V}_{k-1,k+1}(\Omega)$, using the discrete $H^1$-product in $W_k(K)$ discretizes $(\nabla\div\boldsymbol{v}_h, \nabla\div\boldsymbol{w}_h)_K$ as follows 
\begin{align*}
	a_h^K(\boldsymbol{v}_h, \boldsymbol{w}_h):=[\div \boldsymbol{v}_h, \div \boldsymbol{w}_h ]_{n,K}&=(\nabla\Pi_k^{\nabla,K}\operatorname{div}\boldsymbol{v}_h, \nabla\Pi_k^{\nabla,K}\operatorname{div}\boldsymbol{w}_h)_K\\
	&\quad+S_n^K((I-\Pi_k^{\nabla,K})\operatorname{div}\boldsymbol{v}_h, (I-\Pi_k^{\nabla,K})\operatorname{div}\boldsymbol{w}_h).
\end{align*}
By the standard argument \cite{VEM2013} and the definition of $S^K_n(\cdot,\cdot)$ in Remark \ref{remark2}, the local bilinear form $a_h^K(\cdot, \cdot)$ satisfies the following properties:\\
$\bullet$ consistency: for all $\v_h\in \boldsymbol{V}_{k-1,k+1}(K)$ and  $\boldsymbol{q}_{k+1}\in \boldsymbol{P}_{k+1}(K)$, 
\begin{equation}\label{ahConsist}
	a_h^K(\v_h,\boldsymbol{q}_{k+1}) = (\nabla\div \v_h, \nabla\div \boldsymbol{q}_{k+1})_K,
\end{equation}
$\bullet$ stability: for all $\v_h\in \boldsymbol{V}_{k-1,k+1}(K)$, 
\begin{equation}\label{ahStab}
	(\nabla\div \boldsymbol{v}_h,\nabla\div\boldsymbol{v}_h)_K \lesssim a_h^K(\boldsymbol{v}_h,\v_h)\lesssim  (\nabla\div\v_h,\nabla\div\v_h)_K.
\end{equation}
\par
For the bilinear form $(\v_h,\operatorname*{curl}\boldsymbol{\phi}_h )_K$ with $\boldsymbol{\phi}_h\in\boldsymbol{\Sigma}_{0,k}(K)$ and  $\boldsymbol{v}_h\in\boldsymbol{V}_{k-1,k+1}(K)$, we use $b_h^K(\cdot,\cdot)$ defined in \eqref{bhk} to discretize it. 
From the lower bound estimate of Lemma \ref{refLema4}, for any $\boldsymbol{v}_h,\boldsymbol{w}_h\in\boldsymbol{V}_{k-1,k+1}(K)$, we easily obtain the coercivity 
\begin{equation}\label{bhCoer}
	(\boldsymbol{v}_h,\boldsymbol{v}_h)_K\lesssim b_h^K(\boldsymbol{v}_h,\boldsymbol{v}_h).
\end{equation}
Defining a scaled norm $$ \|\v_h\|^2_{h,\boldsymbol{V}(K)}:= \|\vh\|^2_K+h_K^2\|\div\v_h\|^2_K+h_K^4|\div\v_h|^2_{1,K}$$ and using the upper bound estimate of Lemma \ref{refLema4} and the polynomial inverse estimate, we have the continuity   
\begin{equation}
	\begin{aligned}\label{bhConti}
		b_h^K(\boldsymbol{v}_h,\boldsymbol{w}_h) &\lesssim (\|\vh\|^2_K+h_K^2\|\div(I-\Pi_{k-1}^{0,K})\vh\|^2_K+h_K^4|\div(I-\Pi_{k-1}^{0,K})\v_h|^2_{1,K})^{\frac{1}{2}} \\
		&\quad\quad(\|\wh\|^2_K+h_K^2\|\div(I-\Pi_{k-1}^{0,K})\w_h\|^2_K+h_K^4|\div(I-\Pi_{k-1}^{0,K})\w_h|^2_{1,K})^{\frac{1}{2}}\\
		&\lesssim  \|\vh\|_{h,\boldsymbol{V}(K)} \|\wh\|_{h,\boldsymbol{V}(K)}.
	\end{aligned}
\end{equation}
The consistency for $b_h^K(\cdot,\cdot)$ is satisfied by
\begin{equation}\label{bhConsist}
	b_h^K(\boldsymbol{w}_h,\boldsymbol{q}_{k-1})=(\boldsymbol{w}_h,\boldsymbol{q}_{k-1})_K,\quad \forall \boldsymbol{w}_h\in \boldsymbol{V}_{k-1,k+1}(K),\boldsymbol{q}_{k-1}\in\boldsymbol{P}_{k-1}(K).
\end{equation}
\par
At last, according to the discrete $L^2$-product in $\boldsymbol{\Sigma}_{0,k}(K)$, we define the bilinear form by 
\begin{align*}
	c_h^K(\boldsymbol{\phi}_h, \boldsymbol{\psi}_h):=[\boldsymbol{\phi}_h, \boldsymbol{\psi}_h]_{e,K} = (\Pi^{0,K}_0\boldsymbol{\phi}_h, \Pi^{0,K}_0\boldsymbol{\psi}_h)+S_e^K((I-\Pi^{0,K}_0)\boldsymbol{\phi}_h, (I-\Pi^{0,K}_0)\boldsymbol{\psi}_h), \\
	\forall \boldsymbol{\phi}_h, \boldsymbol{\psi}_h\in \boldsymbol{\Sigma}_{0, k}(K),
\end{align*}
which is used to discretize $(\nabla q_h, \boldsymbol{\phi}_h )_K$ for all $q_h\in U_1(K)$ and $\boldsymbol{\phi}_h \in \boldsymbol{\Sigma}_{0,k}(K)$.
We also have  \\
$\bullet$ consistency 
\begin{equation}\label{chConsist}
	c_h^k(\boldsymbol{\phi}_h,\boldsymbol{q}_0)= (\boldsymbol{\phi}_h,\boldsymbol{q}_0)_K,\quad \forall \boldsymbol{\phi}_h\in\boldsymbol{\Sigma}_{0,k}(K),\boldsymbol{q}_0\in\boldsymbol{P}_0(K).
\end{equation}
$\bullet$ stability 
\begin{equation}\label{chStab}
	(\boldsymbol{\phi}_h,\boldsymbol{\phi}_h)_K\lesssim c_h^K(\boldsymbol{\phi}_h,\boldsymbol{\phi}_h)\lesssim (\boldsymbol{\phi}_h,\boldsymbol{\phi}_h)_K,\quad \forall \boldsymbol{\phi}_h \in \boldsymbol{\Sigma}_{0,k}(K).
\end{equation}
\par
As usual,  the global bilinear forms $a_h(\cdot, \cdot) $,  $b_h(\cdot, \cdot) $ and $ c_h(\cdot, \cdot) $ are defined by 
\begin{equation*}
	\begin{aligned}
		a_h(\boldsymbol{v}, \boldsymbol{w})&=\sum_{K\in \mathcal{T}_h} a_h^K(\boldsymbol{v}, \boldsymbol{w}),  \quad \forall \boldsymbol{v} , \w \in \boldsymbol{V}_{k-1, k+1}(\Omega),  \\
		b_h(\boldsymbol{v}, \operatorname{curl} \boldsymbol{\phi}) &=\sum_{K\in \mathcal{T}_h}b_h^K(\boldsymbol{v}, \operatorname{curl}\boldsymbol{\phi}),  \quad  \forall \boldsymbol{v}\in \boldsymbol{V}_{k-1, k+1}(\Omega), \boldsymbol{\phi} \in \boldsymbol{\Sigma}_{0,k}(\Omega), \\
		c_h(\boldsymbol{\phi}, \nabla q)&= \sum_{K\in \mathcal{T}_h}c_h^K(\boldsymbol{\phi}, \nabla q), \quad \forall \boldsymbol{\phi} \in \boldsymbol{\Sigma}_{0,k}(\Omega), q\in U_1(\Omega).
	\end{aligned}
\end{equation*}
\subsection{The discrete problem}
By imposing homogeneous boundary conditions, we define the following discrete spaces:
\begin{gather*}
	U_h:= U_{1}(\Omega)\cap H_0^1(\Omega), \quad\boldsymbol{\Sigma}_h:=\boldsymbol{\Sigma}_{0,k+1}(\Omega)\cap \boldsymbol{V}_0(\Omega), \\\boldsymbol{V}_h:=\boldsymbol{V}_{k-1,k+1}(\Omega)\cap \boldsymbol{V}_0(\Omega), \quad W_h:=W_{k}(\Omega)\cap H_0^1(\Omega)\cap L_0^2(\Omega).
\end{gather*}
These spaces form an exact discrete complex:
\begin{equation}\label{homdisComplex}
	0 \stackrel{}{\longrightarrow} U_h\stackrel{\nabla}{\longrightarrow} \boldsymbol{\Sigma}_{h} \stackrel{\curl }{\longrightarrow} \boldsymbol{V}_{h} \stackrel{\div }{\longrightarrow} W_{h}\longrightarrow 0.
\end{equation}
Based on the above preparations, we are ready to state the virtual element scheme of \eqref{mixForm}: find $(\boldsymbol{u}_h, \boldsymbol{\varphi}_h, p_h)\in \boldsymbol{V}_h \times \boldsymbol{\Sigma}_{h}\times U_h$ such that
\begin{equation}\label{disProblem}
	\left\{\begin{aligned}
		a_{h}\left(\boldsymbol{u}_{h},  \boldsymbol{v}_{h}\right)+b_{h}\left(\operatorname{curl} \boldsymbol{\varphi}_{h},\boldsymbol{v}_{h}\right) & =\left(\boldsymbol{f}_{h},   \boldsymbol{v}_{h}\right),  \quad \forall \boldsymbol{v}_{h} \in \boldsymbol{V}_h,  \\
		b_{h}\left(\boldsymbol{u}_{h},  \operatorname{curl} \boldsymbol{\phi}_{h}\right)  +c_h(\nabla p_h, \boldsymbol{\phi}_h) &=0,  \quad \forall \boldsymbol{\phi}_{h} \in \boldsymbol{\Sigma}_h, \\
		c_h(\boldsymbol{\varphi}_h,\nabla q_h)&=0, \quad \forall q_h \in U_h,
	\end{aligned}\right.
\end{equation}
where $\boldsymbol{f}_h|_K = \Pi^{0,K}_{k-1} \boldsymbol{f}$ satisfies 
\begin{equation}\label{rightHandError}
	\|\boldsymbol{f}-\boldsymbol{f}_h\|\lesssim h^{s}\|\boldsymbol{f}\|_{s}, \quad 0<s\le k.
\end{equation}
\par 
We introduce the subspaces 	
\begin{align*}
	\boldsymbol{X}_h:&=\left\{\boldsymbol{\phi}_h\in\boldsymbol{\Sigma}_{h}: c_h(\boldsymbol{\phi}_h,\nabla q_h)=0, \forall q_h\in U_h\right\}, \\
	\boldsymbol{Y}_h:&= \left\{\boldsymbol{v}_h\in \boldsymbol{V}_{h} :b_{h}(\boldsymbol{v}_{h},  \operatorname{curl} \boldsymbol{\phi}_{h}) =0, \forall \boldsymbol{\phi}_h \in \boldsymbol{\Sigma}_{h} \right\}.
\end{align*}
From stability \eqref{chStab} and coercivity \eqref{bhCoer}, we have the following orthogonal decompositions with respect to discrete inner products $c_h(\cdot,\cdot)$ and $b_h(\cdot,\cdot)$:
\begin{align}\label{disDec1}
	\boldsymbol{\Sigma}_h&=\nabla U_h \oplus^\bot \boldsymbol{X}_h , \\
	\boldsymbol{V}_h &=\curl \boldsymbol{\Sigma}_h \oplus^\bot \boldsymbol{Y}_h.\label{disDec2}
\end{align}
In order to prove the well-posedness of the discrete problem \eqref{disProblem}, we need the following discrete Friedrichs inequalities on $\boldsymbol{X}_h$ and $\boldsymbol{Y}_h$.
\begin{lemma}\label{lemDisFre}
	For all $\boldsymbol{\phi}_h \in \boldsymbol{X}_h$ and $\boldsymbol{v}_h\in \boldsymbol{Y}_h$, there hold
	\begin{align}\label{disFre1}
		\|\boldsymbol{\phi}_h\|&\lesssim  \|\operatorname*{curl} \boldsymbol{\phi}_h\|, \\
		\|\boldsymbol{v}_h\| &\lesssim \|\operatorname*{div}\boldsymbol{v}_h\|_1. \label{disFre2}
	\end{align}
\end{lemma}
\begin{proof}
	Taking  \eqref{disFre2} as an example, we provide a detailed proof.  
	For any given $\boldsymbol{v}_h  \in \boldsymbol{Y}_h$, there exists a unique solution  $\boldsymbol{\rho} \in \boldsymbol{H}(\operatorname{curl}; \Omega) $ satisfying weakly
	\begin{equation}
		\operatorname{curl} \operatorname{curl} \boldsymbol{\rho}=\operatorname{curl} \boldsymbol{v}_h,  \quad \operatorname{div} \boldsymbol{\rho}=0 \quad \text{in } \Omega,  
		\quad \boldsymbol{\rho} \wedge \boldsymbol{n}=0  \quad \text{on } \Gamma. 
	\end{equation}
	Let $\boldsymbol{w} = \boldsymbol{v}_h - \operatorname{curl} \boldsymbol{\rho} $.  Then $\boldsymbol{w}$ weakly satisfies
	\begin{equation*}
		\operatorname{div} \boldsymbol{w} =\operatorname{div}\boldsymbol{v}_h, \quad \operatorname{curl} \boldsymbol{w} =0 \quad \text{in } \Omega,  \quad \boldsymbol{w} \cdot \boldsymbol{n} =0 \quad \text{on } \Gamma,
	\end{equation*}
	which implies $\w \in \boldsymbol{Y}(\Omega)$. By the Friedrichs inequality \eqref{Frieq3}, it follows that
	\begin{equation}\label{aux6}
		\|\boldsymbol{w}\|_{s}\lesssim  \|\operatorname{div}\boldsymbol{w}\|=\|\operatorname{div}\v_h\|.
	\end{equation}
	Hence, $\boldsymbol{I}_h \boldsymbol{w}$ is well-defined, and the commutativity between $\boldsymbol{I}_h$ and $J_h$ in diagram \eqref{CommutDigram} establishes that
	\begin{equation*}
		\div(\boldsymbol{I}_h\boldsymbol{w} -\boldsymbol{v}_h) = J_h \div \boldsymbol{w} -\div\v_h=J_h \div\boldsymbol{v}_h -\div\v_h =0.
	\end{equation*}
	Combined with the exactness of the discrete complex \eqref{homdisComplex}, there exists $\boldsymbol{\rho}_h \in \boldsymbol{\Sigma}_h$ satisfying
	\begin{equation}\label{aux1}
		\curl \boldsymbol{\rho}_h = \boldsymbol{v}_h- \boldsymbol{I}_h \boldsymbol{w}.
	\end{equation}
	According to the interpolation error estimate \eqref{Intpolation1} and \eqref{aux6},  we get
	\begin{equation}\label{aux2}
		\|\boldsymbol{I}_h\boldsymbol{w}\| \le \|\boldsymbol{w}\| +\|\boldsymbol{w} -\boldsymbol{I}_h \boldsymbol{w}\|\lesssim  \|\operatorname{div}\boldsymbol{v}_h\|_1.
	\end{equation}
	Using \eqref{aux1} and the fact that $\boldsymbol{v}_h \in \boldsymbol{Y}_h$, we have 
	\begin{equation*}
		b_h(\boldsymbol{v}_h, \boldsymbol{v}_h) = b_h(\boldsymbol{v}_h, \boldsymbol{I}_h\boldsymbol{w}+\operatorname{curl}\boldsymbol{\rho}_h)=b_h(\boldsymbol{v}_h, \boldsymbol{I}_h\boldsymbol{w}).
	\end{equation*}
	From coercivity \eqref{bhCoer} and continuity \eqref{bhConti}, it follows that  
	\begin{equation*}
		\|\boldsymbol{v}_h\|^2\lesssim b_h(\boldsymbol{v}_h, \boldsymbol{v}_h)=b_h(\boldsymbol{v}_h, \boldsymbol{I}_h\boldsymbol{w})\lesssim  \|\boldsymbol{v}_h\|_{\boldsymbol{V}(\Omega)}\|\boldsymbol{I}_h\boldsymbol{w}\|_{\boldsymbol{V}(\Omega)}.
	\end{equation*}
	Combining the above inequality with  \eqref{aux2} yields 
	\begin{equation*}
		\|\boldsymbol{v}_h\|^2\lesssim \|\boldsymbol{v}_h\|_{\boldsymbol{V}(\Omega)}\|\operatorname{div}\boldsymbol{v}_h\|_1.
	\end{equation*}
	which implies \eqref{disFre2}. \par 
	Due to the coercivity of $c_h(\cdot, \cdot)$ and the interpolation error estimate for $\boldsymbol{I}_h^e$ in  \cite[Theorem 4.5]{VEMforGener}, following the proof of \cite[Lemma 4.5]{ZJQQuadCurl}, we can get \eqref{disFre1} with obvious extension from two dimensions to three dimensions.
	The proof is complete.
\end{proof}
\par 
Now we present the main result of this subsection.
\begin{theorem}\label{ExiUni}
	The discrete problem \eqref{disProblem} has a unique solution $\boldsymbol{u}_h$ and $\boldsymbol{\varphi}_h$ with $p_h=0$.
\end{theorem}
\begin{proof}
	Following the proof of Theorem \ref{Theorem1}, we set $ 
	B_h\left((\boldsymbol{v}_h, q_h), \boldsymbol{\phi}_h\right)=b_h(\boldsymbol{v}_h, \operatorname{curl}\boldsymbol{\phi}_h)+c_h(\nabla q_h, \boldsymbol{\phi}_h)$ 
	and introduce the space
	\begin{equation*}
		\boldsymbol{Z}_h=\{(\boldsymbol{v}_h, q_h) \in \boldsymbol{V}_h\times U_h:B_h\left((\boldsymbol{v}_h, q_h), \boldsymbol{\phi}_h\right) =0, \forall \boldsymbol{\phi}_h \in \boldsymbol{\Sigma}_h\}.
	\end{equation*}
	Any $(\boldsymbol{v}_h, q_h) \in \boldsymbol{Z}_h$ can be identified with $\boldsymbol{v}_h \in \boldsymbol{Y}_h$ and $q_h = 0$. Indeed, choosing $\boldsymbol{\phi}_h = \nabla q_h$ yields $B_h\left((\boldsymbol{v}_h, q_h), \boldsymbol{\phi}_h\right) = c_h(\nabla q_h, \nabla q_h) = 0$, whence, by the stability \eqref{chStab} and the Poincar\'{e} inequality, we obtain $q_h = 0$.\par
	Setting $A_h\left((\boldsymbol{u}_h, p_h), (\boldsymbol{v}_h, q_h)\right) := a_h(\boldsymbol{u}_h, \boldsymbol{v}_h)$, we rewrite \eqref{disProblem} as
	\begin{equation*}
		\begin{aligned}
			A_h((\boldsymbol{u}_h,p_h),(\boldsymbol{v}_h,q_h))+B_h((\boldsymbol{v}_h,q_h),\boldsymbol{\varphi}_h) &= (\boldsymbol{f}_h, \boldsymbol{v}_h),\\
			&\quad \forall (\boldsymbol{v}_h,q_h)\in \boldsymbol{V}_h\times U_h,\\
			B_h((\boldsymbol{u}_h,p_h),\boldsymbol{\phi}_h)&=0,\quad \forall \boldsymbol{\phi}_h \in \boldsymbol{\Sigma}_h.
		\end{aligned}
	\end{equation*}
	According to the stability \eqref{ahStab}, the Poincar\'{e} inequality, and Lemma \ref{lemDisFre}, we obtain the coercivity of $a_h(\cdot,\cdot)$ on $\boldsymbol{Y}_h$ 
	\begin{equation}\label{ahKerCoer}
		\begin{aligned}
			a_h(\boldsymbol{v}_h,\boldsymbol{v}_h)&\gtrsim (\nabla \div \boldsymbol{v}_h,\nabla\div\boldsymbol{v}_h) \gtrsim \|\div\boldsymbol{v}_h\|^2_1 \\
			& \gtrsim \|\boldsymbol{v}_h\|^2_{\boldsymbol{V}(\Omega)}, \quad \forall \boldsymbol{v}_h\in \boldsymbol{Y}_h,
		\end{aligned}
	\end{equation}
	which implies the coercivity of $A_h(\cdot,\cdot)$ on $\boldsymbol{Z}_h$ 
	\begin{equation}
		\begin{aligned}
			A_h\left((\boldsymbol{v}_h, q_h), (\boldsymbol{v}_h, q_h)\right) &= a_h(\boldsymbol{v}_h,\boldsymbol{v}_h)
			\gtrsim \|\boldsymbol{v}_h\|^2_{\boldsymbol{V}(\Omega)}\\
			&=\|\boldsymbol{v}_h\|_{\boldsymbol{V}(\Omega)}^2+\|q_h\|^2_1, \quad \forall (\boldsymbol{v}_h, q_h)\in 	\boldsymbol{Z}_h.
		\end{aligned}
	\end{equation}
	\par 
	Next, we present the discrete inf-sup condition for $B_h(\cdot, \cdot)$. 
	For any  $\boldsymbol{\phi}_h\in\boldsymbol{\Sigma}_h$,
	from the  decomposition \eqref{disDec1}, there exist  $\lambda_h \in U_h$ and $\boldsymbol{z}_h\in \boldsymbol{X}_h$  such that 
	\begin{equation*}
		\boldsymbol{\phi}_h = \nabla \lambda_h + \boldsymbol{z}_h.
	\end{equation*} 
	Then, taking $q_h =\lambda_h, \boldsymbol{v}_h=\operatorname{curl} \boldsymbol{\phi}_h$, using the coercivity \eqref{bhCoer}, the stability \eqref{chStab}, the discrete Friedrichs inequality \eqref{disFre1} and the Poincar\'{e} inequality, we have
	\begin{equation*}
		\begin{aligned}
			B_h\left((\boldsymbol{v}_h, q_h), \boldsymbol{\phi}_h\right) &=b_h(\operatorname{curl}\boldsymbol{\phi}_h, \operatorname{curl}\boldsymbol{\phi}_h )+c_h(\nabla \lambda_h, \boldsymbol{\phi}_h)\\
			&=b_h(\operatorname{curl}\boldsymbol{z}_h, \operatorname{curl}\boldsymbol{z}_h)+ c_h(\nabla \lambda_h, \nabla \lambda_h)\\
			&\gtrsim  (\operatorname{curl}\boldsymbol{z}_h, \operatorname{curl}\boldsymbol{z}_h)+(\nabla \lambda_h, \nabla \lambda_h) \\
			& \gtrsim \|\boldsymbol{z}_h\|^2_{\boldsymbol{H}(\operatorname*{curl};\Omega)}+\|\nabla\lambda_h\|^2_{\boldsymbol{H}(\operatorname*{curl};\Omega)}\\
			&\ge\|\boldsymbol{\phi}_h\|^2_{\boldsymbol{H}(\operatorname*{curl};\Omega)}.
		\end{aligned}
	\end{equation*}
	Thus, the coercivity on $\boldsymbol{Z}_h$ 
	and Babu$\check{\text{s}}$ka-Brezzi
	condition are satisfied, which implies that \eqref{disProblem} has a unique solution. \par 
	Finally, taking $\boldsymbol{\phi}_h =\nabla p_h$ in the second equation of \eqref{disProblem}, we get $c_h(\nabla p_h, \nabla p_h) = 0$. The stability \eqref{chStab} and the Poincar\'{e} inequality lead to $p_h=0$. The proof is complete.  
\end{proof}
\begin{remark}\label{RemarkPhi}
	While the continuous solution satisfies $\boldsymbol{\varphi} = 0$, its discrete counterpart $\boldsymbol{\varphi}_h$ obtained from the discrete scheme \eqref{disProblem} is  nonzero.
	This discrepancy arises because, in the context of VEM,  the right-hand side term $(\boldsymbol{f}, \boldsymbol{v}_h)$ cannot be computed exactly. Consequently, the curl-free field $\boldsymbol{f}$ is replaced by its projection $\boldsymbol{f}_h$, which may not preserve the curl-free property. This  leads to the discrete norm of $\boldsymbol{\varphi}_h$
	\begin{equation*}
		b_h(\curl\boldsymbol{\varphi}_h,\curl\boldsymbol{\varphi}_h)=(\boldsymbol{f}_h,\curl\boldsymbol{\varphi}_h) \ne 0,
	\end{equation*} 
	by taking  $\boldsymbol{v}_h = \operatorname{curl} \boldsymbol{\varphi}_h$ in the first equation of \eqref{disProblem}.  
	\par
	Alternatively, suppose that $\boldsymbol{f} = -\nabla j$ is explicitly known and let $j_h$ be an appropriate polynomial projection of $j$. Then $-\nabla j_h$ serves as an approximation to $\boldsymbol{f}$. By reformulating the right-hand side as $(j_h, \operatorname{div} \boldsymbol{v}_h)$ and testing  the first equation of \eqref{disProblem} with $\boldsymbol{v}_h = \operatorname{curl} \boldsymbol{\varphi}_h$ , we have
	\begin{equation*}
		\|\operatorname{curl}\boldsymbol{\varphi}_h\|^2
		\lesssim b_h(\curl\boldsymbol{\varphi}_h,\curl\boldsymbol{\varphi}_h)=(j_h,\div\operatorname{curl}\boldsymbol{\varphi}_h) =0,
	\end{equation*}
	which, together with \eqref{disFre1}, yields $\boldsymbol{\varphi}_h = 0$.
\end{remark}

\begin{remark}\label{disInfSup}
	From the proof of Theorem \ref{ExiUni}, we also have the discrete inf-sup condition for $b_h(\cdot,\cdot)$.
	There exists a positve constant $\beta $ independent of $h$ such that 
	\begin{equation}\label{disLBB}
		\sup _{\boldsymbol{v}_{h} \in \boldsymbol{V}_h /\{0\}} \frac{b_h\left(\boldsymbol{v}_h, \operatorname{curl} \boldsymbol{\phi}_h\right)}{\left\|\boldsymbol{v}_{h}\right\|_{\boldsymbol{V}( \Omega)}} \geq \beta\|\operatorname{curl}\boldsymbol{\phi}_{h}\|,  \quad \forall \boldsymbol{\phi}_{h} \in \boldsymbol{X}_h.
	\end{equation}
\end{remark}
\subsection{Convergence analysis}
By using the standard Dupont-Scott theory \cite{Polyesti}, we have the following local approximation results.
\begin{lemma}
	Assume that the polyhedron $K$ satisfies the regularity assumptions $\boldsymbol{(A1)}$-$\boldsymbol{(A3)}$. For all $\v \in \boldsymbol{H}^s(\Omega)$ and $\div\v \in H^{s+1}(\Omega)$,
	there exist  $\v^{\pi}_{k-1}\in \boldsymbol{P}^{\text{dc}}_{k-1}(\Omega)$ and $ \v^{\pi}_{k+1} \in \boldsymbol{P}^{\text{dc}}_{k+1}(\Omega)$ with $s\le k$ such that 
	\begin{align}
		\label{polyappro}
		|\v -\v^{\pi}_{k-1}|_{m, K}&\lesssim h_K^{l-m}|\v|_{l, K}, \quad 0\le m\le l\le s, \\
		\label{divpolyappro}
		|\operatorname{div}(\v-\v^{\pi}_{k+1})|_{m, K}&\lesssim h_K^{l-m}|\operatorname{div}\v|_{l, K}, \quad 0\le m \le  l\le s+1,
	\end{align}
	where  $\boldsymbol{P}_k^{\text{dc}}(\Omega) = \left\{\boldsymbol{v}\in \boldsymbol{L}^2(\Omega); \boldsymbol{v}|_K \in \boldsymbol{P}_k(K), \forall K \in T_h \right\}$.
\end{lemma}
\begin{theorem}\label{Theorem4}
	Suppose that $(\boldsymbol{u}, \boldsymbol{\varphi}, p) \in \boldsymbol{V}_0(\Omega) \times \boldsymbol{H}_0(\curl;\Omega)\times H^1_0(\Omega) $ is the solution of the  problem \eqref{mixForm} with $\boldsymbol{\varphi} =0$, $p=0$, 
	and  $(\boldsymbol{u}_h, \boldsymbol{\varphi}_h, p_h) \in \boldsymbol{V}_h\times \boldsymbol{ \Sigma}_h\times U_h$  the solution of the discrete scheme \eqref{disProblem} with  $p_h=0$. There holds 
	\begin{equation*}
		\begin{aligned}
			\|\boldsymbol{u}-\boldsymbol{u}_h\|_{\boldsymbol{V}(\Omega)}\lesssim & \inf _{\boldsymbol{z}_{h} \in \boldsymbol{Y}_{h}}\left\|\boldsymbol{u}-\boldsymbol{z}_{h}\right\|_{\boldsymbol{V}(\Omega)} \\
			& +\inf _{\boldsymbol{v}^{\pi}_{k+1} \in \boldsymbol{P}_{k+1}^{\mathrm{dc}}(\Omega)}\left|\div\left(\boldsymbol{u}-\boldsymbol{v}^{\pi}_{k+1}\right)\right|_1+\left\|\boldsymbol{f}-\boldsymbol{f}_{h}\right\|. 
		\end{aligned}
	\end{equation*}
\end{theorem}
\begin{proof}
	For any $\boldsymbol{z}_h\in \boldsymbol{Y}_h $ and $\boldsymbol{v}^{\pi}_{k+1} \in \boldsymbol{P}^{dc}_{k+1}(\Omega)$, setting $\boldsymbol{\delta}_h = \boldsymbol{z}_h -\boldsymbol{u}_h \in \boldsymbol{Y}_h$, we obtain 
	\begin{equation*}
		\begin{aligned}
			\|\boldsymbol{z}_h& -\boldsymbol{u}_h\|^2_{\boldsymbol{V}(\Omega)}\\
			&\lesssim  a_h(\boldsymbol{z}_h-\boldsymbol{u}_h, \boldsymbol{\delta}_h)
			\quad (\text{use the  coercivity } \eqref{ahKerCoer})\\
			&=\sum_{K\in \mathcal{T}_h}(a_h^K(\boldsymbol{z}_h-\boldsymbol{v}^{\pi}_{k+1}, \boldsymbol{\delta}_h)+a_h^K(\boldsymbol{v}^{\pi}_{k+1}, \boldsymbol{\delta}_h)) \nonumber \\
			&\quad -(\boldsymbol{f}_h, \boldsymbol{\delta}_h) \quad (\text{use } \eqref{disProblem} \text{ with } b_h(\operatorname{curl}\boldsymbol{\varphi}_h,\boldsymbol{ \delta}_h)=0)\\
			&=\sum_{K\in \mathcal{T}_h}\left(a_h^K(\boldsymbol{z}_h-\boldsymbol{v}^{\pi}_{k+1}, \boldsymbol{\delta}_h)+(\nabla\div\boldsymbol{v}^{\pi}_{k+1}, \nabla\div\boldsymbol{\delta}_h)_K\right) \nonumber\\
			&\quad-(\boldsymbol{f}_h, \boldsymbol{\delta}_h)\quad (\text{use the consistency } \eqref{ahConsist})\\
			&=\sum_{K\in \mathcal{T}_h}\left(a_h^K(\boldsymbol{z}_h-\boldsymbol{v}^{\pi}_{k+1}, \boldsymbol{\delta}_h)+(\nabla\div(\boldsymbol{v}^{\pi}_{k+1}-\boldsymbol{u}), \nabla\div\boldsymbol{\delta}_h)_K\right)\nonumber \\
			&\quad+(\boldsymbol{f}-\boldsymbol{f}_h, \boldsymbol{\delta}_h) \quad (\text{use } \eqref{equivalentProblem}).
		\end{aligned}
	\end{equation*}
	It follows from the stability \eqref{ahStab} that 
	\begin{equation*}
		\|\boldsymbol{z}_h-\boldsymbol{u}_h\|_{\boldsymbol{V}(\Omega)}\lesssim (|\operatorname{div}(\boldsymbol{u-z}_h)|_1+|\operatorname{div}(\boldsymbol{u}-\boldsymbol{v}^{\pi}_{k+1})|_1+\|\boldsymbol{f}-\boldsymbol{f}_h\|).
	\end{equation*} 
	Using the triangle inequality, we get the desired result.
\end{proof}
\begin{theorem}\label{allNormError}
	Under the assumptions of Theorem \ref{Theorem4}. 
	If $\boldsymbol{f}\in \boldsymbol{H}^{s}(\Omega)$ and $\boldsymbol{u} \in \boldsymbol{H}^{s}(\Omega)$ with $\operatorname{div}\boldsymbol{u} \in H^{s+1}(\Omega)$,  $\frac{1}{2} < s \le k$, we have 
	\begin{equation}\label{theor5}
		\|\boldsymbol{u}-\boldsymbol{u}_h\|_{\boldsymbol{V}(\Omega)} \lesssim h^{s}(\|\boldsymbol{u}\|_{s}+\|\operatorname{div}\boldsymbol{u}\|_{s+1}+\|\boldsymbol{f}\|_{s}),
	\end{equation}
	\begin{equation}
		\|\boldsymbol{ \varphi}_h\|\lesssim 	\|\operatorname{curl}\boldsymbol{ \varphi}_h\|\lesssim 
		h^s \|\boldsymbol{f}\|_s.
	\end{equation}
\end{theorem}
\begin{proof}
	Using the decomposition \eqref{disDec1}, we rewrite \eqref{disDec2} as 
	\begin{equation*}
		\boldsymbol{V}_{k-1,k+1}(\Omega) = \curl \boldsymbol{X}_h \oplus^\bot \boldsymbol{Y}_h.
	\end{equation*}
	Then there exists a function $\boldsymbol{\psi}_h \in \boldsymbol{X}_h$  such that $ \boldsymbol{I}_h \bu - \curl \boldsymbol{\psi}_h \in \boldsymbol{Y}_h$.
	On the other hand, by the discrete Friedrichs inequality \eqref{disFre1} on $\boldsymbol{X}_h$, the curl operator is injective.
	Combining the above properties with the coercivity \eqref{bhCoer} of $b_h(\cdot,\cdot)$, we have
	\begin{equation}\label{curlLBB}
		\|\operatorname{curl}\boldsymbol{\psi}_h\|\lesssim \sup_{\boldsymbol{\phi}_{h} \in \boldsymbol{X}_h/\{0\} }\frac{b_h(\curl \boldsymbol{\psi}_h , \operatorname{curl}\boldsymbol{\phi}_h)}{\|\operatorname{curl}\boldsymbol{\phi}_h\|} = \sup_{\boldsymbol{\phi}_{h} \in \boldsymbol{X}_h/\{0\} }\frac{b_h(\boldsymbol{I}_h \boldsymbol{u}, \operatorname{curl}\boldsymbol{\phi}_h)}{\|\operatorname{curl}\boldsymbol{\phi}_h\|}.
	\end{equation}
	Let $\bu^{\pi}_{k-1} \in \boldsymbol{P}_{k-1}^{dc}(\Omega)$ be the approximation to $\boldsymbol{u}$ satisfying \eqref{polyappro}.
	The second equation in \eqref{equivalentProblem} leads to  $(\boldsymbol{u}, \operatorname{curl}\boldsymbol{\phi}_h) =0$ for all $\boldsymbol{\phi}_h \in \boldsymbol{\Sigma}_{0, k}(\Omega)$. Furthermore,
	using consistency \eqref{bhConsist}, we obtain
	\begin{equation*}
		\begin{aligned}
			b_h(\boldsymbol{I}_h\boldsymbol{u}, \operatorname{curl}\boldsymbol{\phi}_h)=(\bu^{\pi}_{k-1}-\boldsymbol{u}, \operatorname{curl}\boldsymbol{\phi}_h)+b_h(\boldsymbol{I}_h\boldsymbol{u}-\bu^{\pi}_{k-1}, \operatorname{curl}\boldsymbol{\phi}_h).
		\end{aligned}
	\end{equation*}
	For the first term on the right-hand side of the above equation, it holds
	\begin{equation}\label{aux7}
		(\bu^{\pi}_{k-1}-\boldsymbol{u}, \operatorname{curl}\boldsymbol{\phi}_h)\lesssim h^s\|\boldsymbol{u}\|_s\|\operatorname{curl}\boldsymbol{\phi}_h\|.
	\end{equation}
	For the second term, according to the continuity \eqref{bhConti}, the definition of the scaled norm $\|\cdot\|_{h,\boldsymbol{V}(K)}$, interpolation errors \eqref{Intpolation1}, \eqref{Intpolation2}, \eqref{Intpolation3}, we get
	\begin{equation*}
		\begin{aligned}
			b_h^K(&\boldsymbol{I}_h\boldsymbol{u}-\bu^{\pi}_{k-1},  \operatorname{curl} \boldsymbol{\phi}_h)\nonumber\\
			&\lesssim \|\boldsymbol{I}_h\boldsymbol{u}-\bu^{\pi}_{k-1}\|_{h,\boldsymbol{V}(K)}\|\operatorname{curl}\boldsymbol{\phi}_h\|_{h,\boldsymbol{V}(K)}\\
			&\lesssim (\|\bu-\boldsymbol{I}_h\bu\|_{h,\boldsymbol{V}(K)}+\|\bu - \bu^{\pi}_{k-1}\|_{h,\boldsymbol{V}(K)})\|\operatorname{curl}\boldsymbol{\phi}_h\|_K\\
			&\lesssim\left(\|\bu-\boldsymbol{I}_h\bu\|_K+h_K\|\div(\bu-\boldsymbol{I}_h\bu)\|_K+h_K^2|\div(\bu-\boldsymbol{I}_h\bu)|_{1,K}\right.\\
			&\quad+ \left. \|\bu-\bu_{k-1}^{\pi}\|_K+h_K|\bu-\bu_{k-1}^{\pi}|_{1,K}+h_K^2|\bu-\bu_{k-1}^{\pi}|_{2,K} \right)\|\operatorname{curl}\boldsymbol{\phi}_h\|_K\\
			&\lesssim h_K^s(\|\boldsymbol{u}\|_{s,K}+h_K^{2}\|\operatorname{div}\boldsymbol{u}\|_{s+1,K})\|\operatorname{curl}\boldsymbol{\phi}_h\|_K,
		\end{aligned}
	\end{equation*}
	which,  together with \eqref{aux7} and \eqref{curlLBB}, yields
	\begin{equation*}
		\|\operatorname{curl}\boldsymbol{\psi}_h\|\lesssim h^s(\|\boldsymbol{u}\|_s+h^{2}\|\operatorname{div}\boldsymbol{u}\|_{s+1}).
	\end{equation*}
	Then we have
	\begin{equation*}
		\begin{aligned}
			\inf _{\boldsymbol{z}_{h} \in \boldsymbol{Y}_{h}}\left\|\boldsymbol{u}-\boldsymbol{z}_{h}\right\|_{\boldsymbol{V}(\Omega)}&\le \|\boldsymbol{u}-\boldsymbol{I}_h\boldsymbol{u}\|_{\boldsymbol{V}(\Omega)}+\|\operatorname{curl}\boldsymbol{\psi}_h\|\\
			&\lesssim  h^s(\|\boldsymbol{u}\|_s+\|\operatorname{div}\boldsymbol{u}\|_{s+1}),
		\end{aligned}
	\end{equation*}
	which, combined with  Theorem \ref{Theorem4}, implies \eqref{theor5}.
	Finally, since $\boldsymbol{f}$ is curl‑free, the coercivity \eqref{bhCoer} gives
	\begin{align*}
		\|\operatorname{curl}\boldsymbol{\varphi}_h\|^2 &\lesssim b_h(\operatorname{curl}\boldsymbol{\varphi}_h,\operatorname{curl}\boldsymbol{\varphi}_h)  = (\boldsymbol{ f}_h,\operatorname{curl}\boldsymbol{\varphi}_h)=(\boldsymbol{f}-\boldsymbol{f}_h,\operatorname{curl}\boldsymbol{\varphi}_h)\\
		&\lesssim \|\boldsymbol{ f}-\boldsymbol{ f}_h\| \|\operatorname{curl}\boldsymbol{ \varphi}_h\| \lesssim h^s\|\boldsymbol{f}\|_s \|\operatorname{curl}\boldsymbol{ \varphi}_h\|,
	\end{align*}
	which, along with \eqref{disFre1}, completes the proof.
\end{proof}
\section{Numerical experiments}
In this section, we present some numerical results for the discrete complex \eqref{discomplex} with $r=k=1$ in three dimensions. 
We consider the quad-div problem \eqref{primalForm} on a unit cube $\Omega = [0,1]^3$, in which the source form $\boldsymbol{f}$ is given such that 
\begin{equation*}
	\boldsymbol{u}(x,y,z)=\nabla\left(x^{3} y^{3} z^{3}(x-1)^{3}(y-1)^{3}(z-1)^{3}\right).
\end{equation*}
\par
We solve the quad-div problem by using the C++ library Vem++ \cite{VemCode}.
Three kinds of meshes are considered as follows.\par
$\bullet $ Cube: structured meshes consisting of cubes; see Fig. \ref{mesh}(a);\par
$\bullet $ Voro: Voronoi tessellations optimized by the Lloyd algorithm; see Fig. \ref{mesh}(b);\par
$\bullet $ Random: Voronoi diagram of a point set randomly displayed inside the domain $\Omega$; see Fig. \ref{mesh}(c).
\begin{figure}[htbp]
	\centering
	\subfigure[Cube]{
		\begin{minipage}[t]{0.3\textwidth}
			\centering
			\includegraphics[scale=0.07]{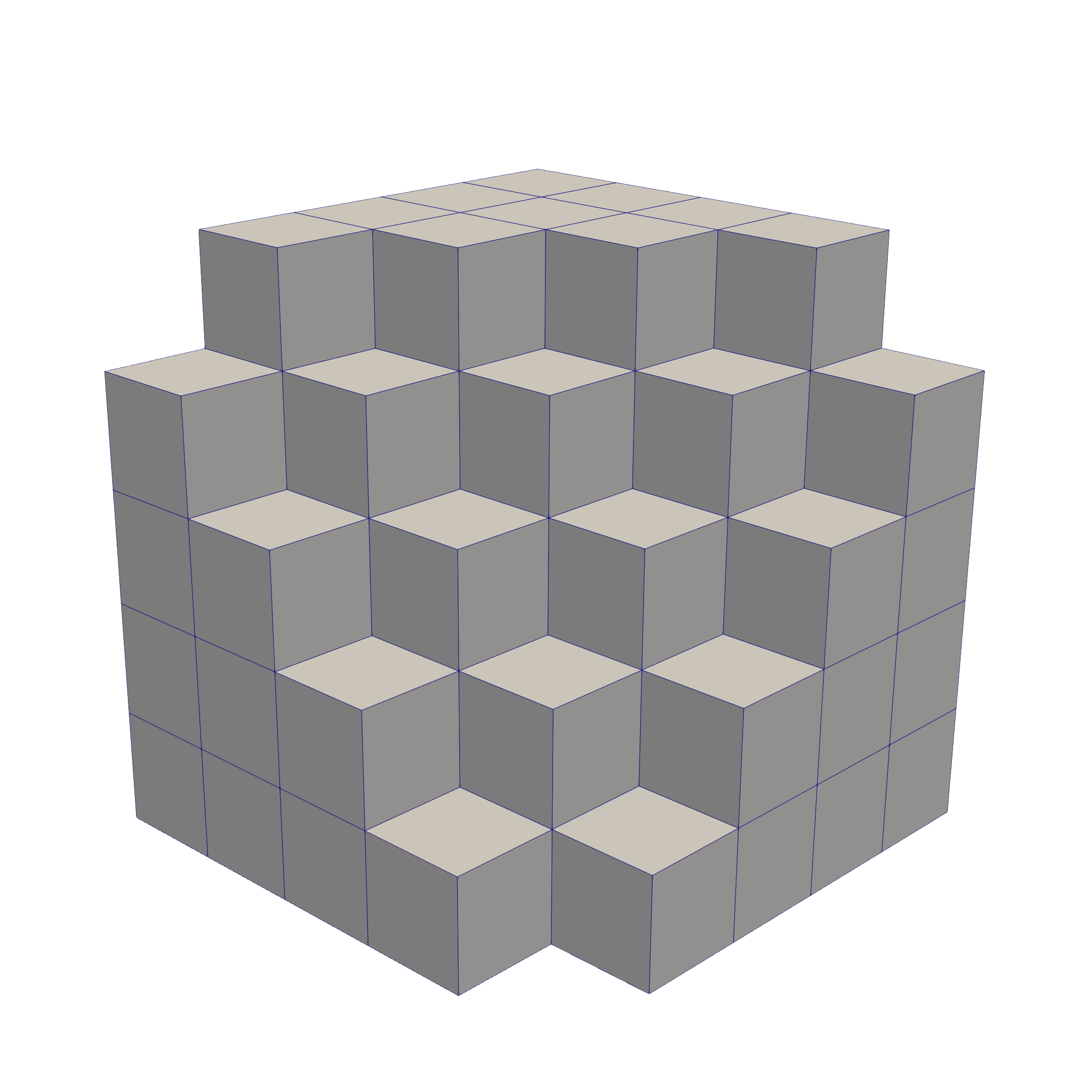}
		\end{minipage}
	}
	\hfill
	\subfigure[Voro]{
		\begin{minipage}[t]{0.3\textwidth}
			\centering
			\includegraphics[scale=0.07]{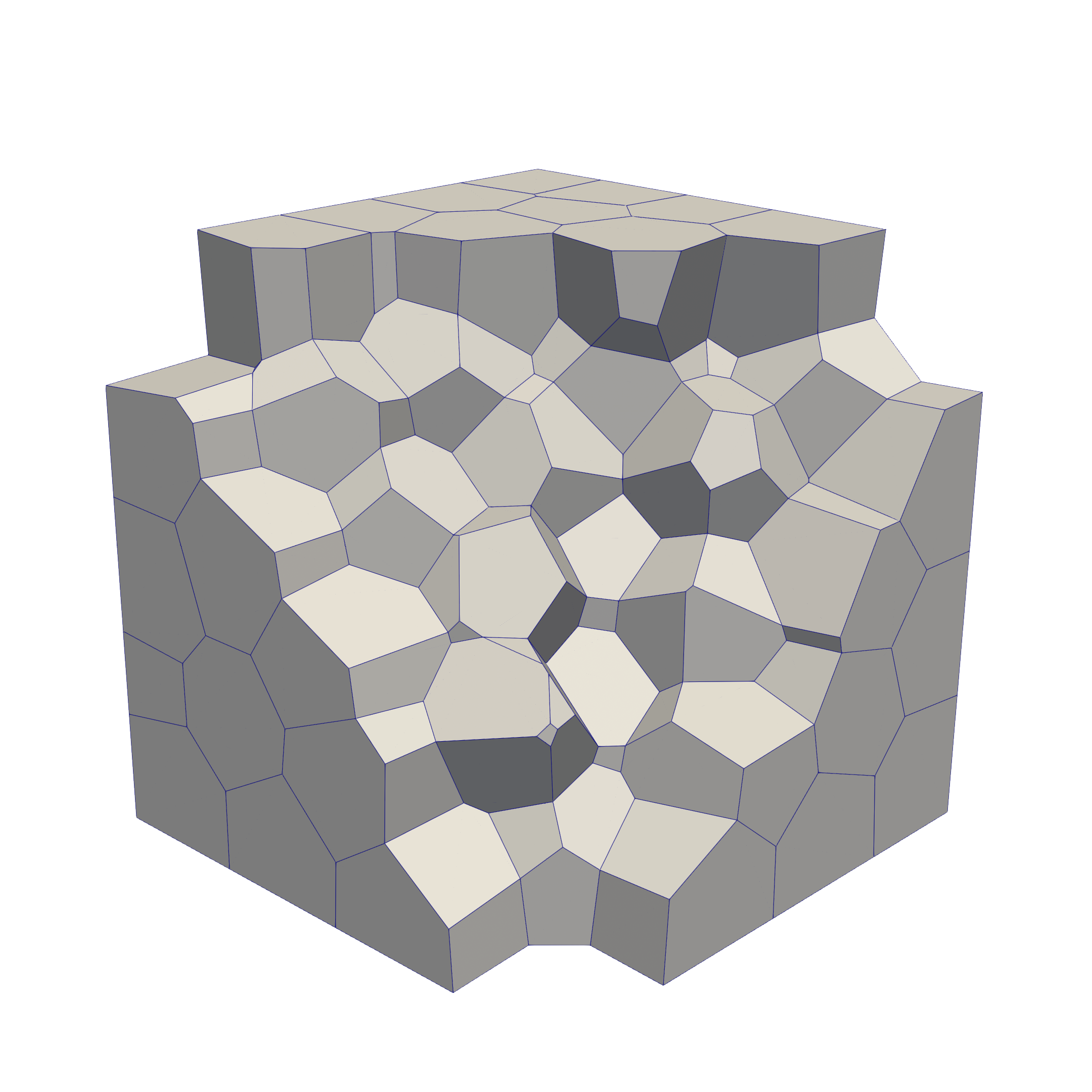}
		\end{minipage}
	}
	\hfill
	\subfigure[Random]{
		\begin{minipage}[t]{0.3\textwidth}
			\centering
			\includegraphics[scale=0.07]{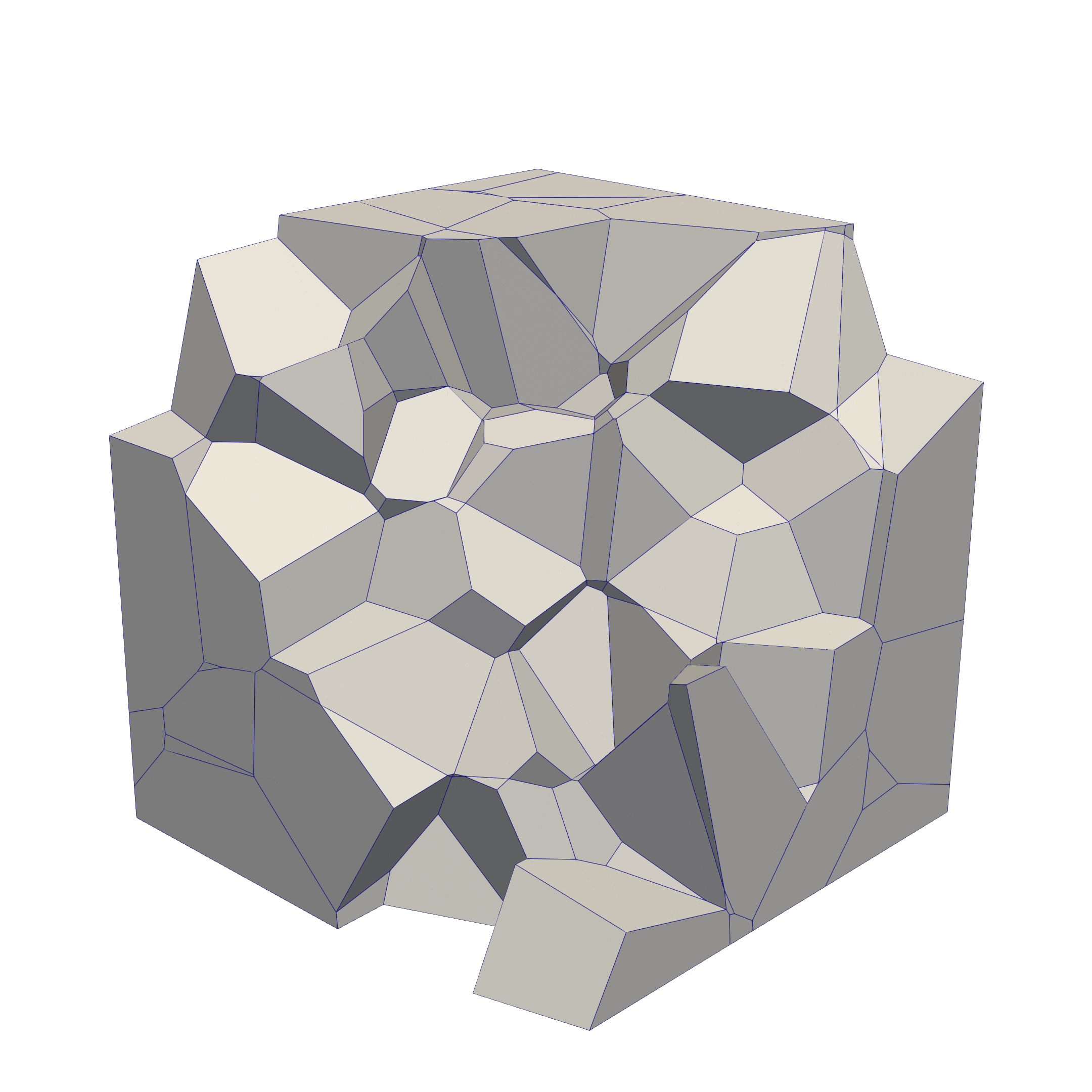}
		\end{minipage}
	}
	\centering
	\caption{Three representatives of the three families of meshes}
	\label{mesh}
\end{figure}
\par
For the computation of the virtual element solution  $\bu_h$, since the error $\bu-\bu_h$ is not directly measured for VEM, 
we instead  define  the discrete error  norm by
\begin{equation*}
	\begin{aligned}
		\|\boldsymbol{e}^{\boldsymbol{u}}\|_h&=\sqrt{b_h(\boldsymbol{I}_h\boldsymbol{u}-\boldsymbol{u}_h, \boldsymbol{I}_h\boldsymbol{u}-\boldsymbol{u}_h)}.
	\end{aligned}
\end{equation*}
In addition, we define two discrete errors associated with the multipliers:
\begin{equation*}
	\|\boldsymbol{e}^{\boldsymbol{\varphi}}\|_h := \sqrt{c_h(\boldsymbol{\varphi}_h, \boldsymbol{\varphi}_h)} \quad 
	\text{and}\quad
	|e^p|_{1,h} := \sqrt{c_h(\nabla p_h, \nabla p_h)}
\end{equation*}
In fact, from the coercivity, continuity, polynomial approximation, and virtual element interpolation estimates, 
the computable error here scales like the ``exact one''.
\par 
We present the convergence results for the lowest order of $r=k = 1$ on the three meshes in Tables~\ref{tab:Cube}, \ref{tab:Voro} and \ref{tab:Rand}. The mesh size $h$ used in our computations is defined as the average element diameter:
\begin{equation*}
	h:=\frac{1}{N}\sum_{K\in \mathcal{T}_h}h_K,
\end{equation*}
where $N$ is the total number of polyhedral elements $K$ in the mesh.
We observe that the convergence orders for the errors $\|\boldsymbol{e}^{\bu}\|_h$ and $\|\boldsymbol{e}^{\boldsymbol{\varphi}}\|_h$  are at least  $O(h)$, while $|e^p|_{1,h}$ remains at the machine precision level, which is consistent with the theoretical results. 
\begin{table}[ht]
	\centering
	\caption{Computed errors and rate of convergence with $r = k = 1$ on Cube meshes}
	\begin{tabular}{rrrrrrr}
		\toprule
		Ndof &	\( h \) & \( \|\boldsymbol{e}^{\boldsymbol{u}}\|_h \) & Rates & \( \|\boldsymbol{e}^{\boldsymbol{\varphi}}\|_h  \) & Rates &\(|e^{p}|_{1,h}\)\\
		\midrule
		790&	0.433012  & 5.146380E-02 &       & 5.990611E-02 &      &6.209179E-16     \\
		5130&	0.216506  & 1.434799E-02 & 1.8427   & 2.290767E-02 & 1.3869 & 4.631082E-14\\
		16094&	0.144337  & 8.368232E-03 & 1.3298  & 1.096638E-02 & 1.8168  & 5.722950E-16\\
		36754&	0.108253  & 6.033595E-03 &  1.1370  & 6.403414E-03 & 1.8701  &4.227228E-16 \\
		119450&	0.072168  & 3.982507E-03 & 1.0246 & 2.922138E-03 & 1.9348  & 2.049788E-14\\
		\bottomrule
	\end{tabular}
	\label{tab:Cube}
\end{table}

\begin{table}[ht]
	\centering
	\caption{Computed errors and rate of convergence with $r = k = 1$ on Voro meshes}
	\begin{tabular}{rrrrrrr}
		\toprule
		Ndof &	\( h \) & \( \|\boldsymbol{e}^{\boldsymbol{u}}\|_h \) & Rates & \( \|\boldsymbol{e}^{\boldsymbol{\varphi}}\|_h  \) & Rates &\(|e^{p}|_{1,h}\)\\
		\midrule
		740&	0.568225  & 7.235101E-02 &       & 1.715414E-01 &  & 1.883983E-15       \\
		3333&	0.318715  & 2.628734E-02 & 1.7510  & 1.287074E-01 &  0.4968  & 1.421457E-13\\
		28523&	0.153136  & 8.690384E-03 & 1.5101   & 5.034928E-02 &   1.2805      & 1.647550E-12\\
		57863&	0.120167  & 6.488065E-03 &  1.2054 & 3.497198E-02 & 1.5032&3.498923E-13  \\
		116983&	0.094650  & 5.037933E-03 & 1.0598 & 2.459276E-02 & 1.4751 & 3.632578E-11  \\
		\bottomrule
	\end{tabular}
	\label{tab:Voro}
\end{table}
\begin{table}[ht]
	\centering
	\caption{Computed errors and rate of convergence with $r = k = 1$ on Random meshes}
	\begin{tabular}{rrrrrrr}
		\toprule
		Ndof &	\( h \) & \( \|\boldsymbol{e}^{\boldsymbol{u}}\|_h \) & Rates & \( \|\boldsymbol{e}^{\boldsymbol{\varphi}}\|_h  \) & Rates &\(|e^{p}|_{1,h}\)\\
		\midrule
		725&	0.682575  & 1.108245E-01 &       & 4.083115E-01 &  & 3.648009E-14       \\
		3698&	0.397259  & 5.013813E-02 & 1.4653   & 2.980492E-01 &  0.5815 & 8.455677E-13\\
		32752&	0.188677  & 1.514060E-02 & 1.6251  & 8.681905E-02 & 1.6740 & 6.558233E-12\\
		
		\bottomrule
	\end{tabular}
	\label{tab:Rand}
\end{table}

\appendix
\section{The dual complex}\label{sec:appendixA}
Based on the generalized Helmholtz decomposition established in \cite{Chen2018}, we characterize the dual space of $\boldsymbol{V}_0(\Omega)$.
The analysis begins with the complex, defined for any $s \in \mathbb{R}$:
\begin{equation}\label{Realscomplex}
	\mathbb{R} \stackrel{\subset}{\longrightarrow} H^{s+3}(\Omega)\stackrel{\nabla}{\longrightarrow} \boldsymbol{H}^{s+2}(\Omega) \stackrel{\curl }{\longrightarrow} \boldsymbol{H}^{s+1}(\Omega)  \stackrel{\div }{\longrightarrow} H^{s}(\Omega)\longrightarrow 0.
\end{equation}    
As shown in \cite{Costabel2010}, this complex is exact on bounded domains that are starlike with respect to a ball.
Consider the orthogonal complement $\boldsymbol{X}(\Omega)$ of $\boldsymbol{H}_0(\curl;\Omega)$ within the complex \eqref{ContiComplex}. The restriction yields a short exact sequence:
\begin{equation*}
	0 \stackrel{}{\longrightarrow}\boldsymbol{X}(\Omega)\stackrel{\nabla \times}{\longrightarrow}\boldsymbol{V}_0(\Omega) \stackrel{\nabla \cdot }{\longrightarrow} H_0^1(\Omega).
\end{equation*}
Furthermore, by \cite[Remark 2.15]{Pauly2016}, the dual complex
\begin{equation}\label{DualShortComplex1}
	H^{-1}(\Omega)\stackrel{\nabla}{\longrightarrow} \boldsymbol{V}'(\Omega)\stackrel{\nabla \times}{\longrightarrow}\boldsymbol{X}'(\Omega)\longrightarrow 0
\end{equation}
is also exact.
Define the space
\begin{equation*}
	\boldsymbol{H}^{-2}(\curl;\Omega):=\left\{\v \in \boldsymbol{H}^{-2}(\Omega): \curl \v \in \boldsymbol{H}^{-1}(\Omega) \right\},
\end{equation*}
equipped with the norm
\begin{equation*}
	\|\v\|^2_{\boldsymbol{H}^{-2}(\curl;\Omega)}:=\|\v\|^2_{-2}+ \|\curl \v\|_{-1}^2.
\end{equation*}
\begin{lemma}
	The complex
	\begin{equation}\label{DualShortComplex2}
		H^{-1}(\Omega)\stackrel{\nabla }{\longrightarrow}\boldsymbol{H}^{-2}(\curl;\Omega) \stackrel{\curl }{\longrightarrow} \boldsymbol{X}'(\Omega) \longrightarrow 0
	\end{equation}
	is exact.
\end{lemma}
\begin{proof}
	Substituting $s = -4$ into the complex \eqref{Realscomplex} yields the identity $\boldsymbol{H}^{-2}(\Omega) \cap \ker(\curl) = \nabla H^{-1}(\Omega)$. Observe that $\boldsymbol{H}^{-2}(\curl;\Omega)\cap \ker(\curl) = \boldsymbol{H}^{-2}(\Omega) \cap \ker(\curl)$. This identity immediately establishes the exactness of the former complex.
	\par
	Moreover, by the Friedrichs inequality \eqref{Frieq2}, the operator $\curl\curl: \boldsymbol{X}(\Omega) \to \boldsymbol{X}'(\Omega)$ is an isomorphism. Hence, $\boldsymbol{X}'(\Omega)=\curl \curl \boldsymbol{X}(\Omega)$, which is contained in $\curl \boldsymbol{H}^{-2}(\curl;\Omega)$. On the other hand, the definition gives $\curl \boldsymbol{H}^{-2}(\curl;\Omega) \subset \boldsymbol{H}^{-1}(\Omega)$. We thus have the inclusions
	$$\curl \boldsymbol{H}^{-2}(\curl;\Omega) \subset\boldsymbol{H}^{-1}(\div;\Omega)=\boldsymbol{H}_0(\curl;\Omega)'\subset \boldsymbol{X}'(\Omega),$$
	where the space $$\boldsymbol{H}^{-1}(\div;\Omega):=\{ \v \in \boldsymbol{H}^{-1}(\Omega): \div \v \in H^{-1}(\Omega)\}$$ is  the dual of $\boldsymbol{H}_0(\curl;\Omega)$ as defined in \cite{Chen2018}. Combining the two inclusions yields the required identity:
	\begin{equation*}
		\boldsymbol{X}'(\Omega)=\curl \boldsymbol{H}^{-2}(\curl;\Omega).
	\end{equation*}
\end{proof}
We have the following commutative diagram:
\[
\begin{tikzcd}
	H^{-1}(\Omega) \arrow[r,"{\nabla }"] 
	& \boldsymbol{H}^{-2}(\curl;\Omega) \arrow[r,"{\curl}"] 
	& \boldsymbol{X}'(\Omega) \arrow[r] & 0 \\
	& \boldsymbol{V}_0(\Omega) \arrow[u,"I"] 
	& \boldsymbol{X}(\Omega) \arrow[l,"{\curl }"'] \arrow[u,"{\curl\curl}" ].
\end{tikzcd}
\]
Combining Corollary 2.5 in \cite{Chen2018} with exact complexes \eqref{DualShortComplex1} and \eqref{DualShortComplex2} gives Lemma \ref{DecLemma}:
\begin{equation*}
	\boldsymbol{V}'(\Omega)= \boldsymbol{H}^{-2}(\curl;\Omega)=\nabla H^{-1}(\Omega)\oplus \curl\boldsymbol{X}(\Omega) = \nabla H^{-1}(\Omega)\oplus \curl \boldsymbol{H}_0(\curl;\Omega).
\end{equation*}

	\bibliographystyle{amsplain}
	\bibliography{4div_reference}

\end{document}